\documentclass[12pt]{amsart}
\usepackage{amsfonts, amssymb, latexsym, hyperref}
\usepackage[usenames, dvipsnames]{color}
\usepackage{ wasysym }
\usepackage{ytableau}
\usepackage{xcolor}
\usepackage{tikz}

\newtheorem{theorem}{Theorem}[section]
\newtheorem{proposition}[theorem]{Proposition}
\newtheorem{lemma}[theorem]{Lemma}

\newtheorem{claim}[theorem]{Claim}
\newtheorem*{claim*}{Claim}
\newtheorem{corollary}[theorem]{Corollary}
\newtheorem{Main Conjecture}[theorem]{Main Conjecture}
\newtheorem{conjecture}[theorem]{Conjecture}
\newtheorem{problem}[theorem]{Problem}
\theoremstyle{remark}

\newtheorem{example}[theorem]{Example}
\newtheorem{Remark}[theorem]{Remark}
\newtheorem{remark}[theorem]{Remark}
\theoremstyle{plain}
\newtheorem{question}[theorem]{Question}

\newcommand\FA{{\mathfrak{D}}} 
\newcommand{\FAring}{\mathfrak{P}}

\newcommand\G{{\sf G}}

\newcommand{\cellsizeL}{19}
\newcommand{\cellsizeS}{14}
\newlength{\cellszL} 
\newsavebox{\cellL}
\sbox{\cellL}{\begin{picture}(\cellsizeL,\cellsizeL)
\put(0,0){\line(1,0){\cellsizeL}}
\put(0,0){\line(0,1){\cellsizeL}}
\put(\cellsizeL,0){\line(0,1){\cellsizeL}}
\put(0,\cellsizeL){\line(1,0){\cellsizeL}}
\end{picture}}
\newcommand\cellifyL[1]{\def\thearg{#1}\def\nothing{}%
\ifx\thearg\nothing
\vrule width0pt height\cellszL depth0pt\else
\hbox to 0pt{\usebox{\cellL} \hss}\fi%
\vbox to \cellszL{
\vss
\hbox to \cellszL{\hss$#1$\hss}
\vss}}
\newcommand\tableauL[1]{\vtop{\let\\\cr
\baselineskip -16000pt \lineskiplimit 16000pt \lineskip 0pt
\ialign{&\cellifyL{##}\cr#1\crcr}}}

\newlength{\cellszS} 
\newsavebox{\cellS}
\sbox{\cellS}{\begin{picture}(\cellsizeS,\cellsizeS)
\put(0,0){\line(1,0){\cellsizeS}}
\put(0,0){\line(0,1){\cellsizeS}}
\put(\cellsizeS,0){\line(0,1){\cellsizeS}}
\put(0,\cellsizeS){\line(1,0){\cellsizeS}}
\end{picture}}
\newcommand\cellifyS[1]{\def\thearg{#1}\def\nothing{}%
\ifx\thearg\nothing
\vrule width0pt height\cellszS depth0pt\else
\hbox to 0pt{\usebox{\cellS} \hss}\fi%
\vbox to \cellszS{
\vss
\hbox to \cellszS{\hss$#1$\hss}
\vss}}
\newcommand\tableauS[1]{\vtop{\let\\\cr
\baselineskip -16000pt \lineskiplimit 16000pt \lineskip 0pt
\ialign{&\cellifyS{##}\cr#1\crcr}}}

\newcommand{\eqssyt}{{\sf EqSSYT}}
\newcommand{\dist}{{\sf dist}}
\newcommand{\xbox}{x}
\newcommand{\factor}{{\sf Eballotfactor}}
\newcommand{\itop}{\tau}  

\newcommand{\excise}[1]{}

\title{Equivariant cohomology, Schubert calculus,\\ and Edge labeled tableaux}
\author{Colleen Robichaux}
\author{Harshit Yadav}
\author{Alexander Yong}
\address{Dept.~of Mathematics, University of Illinois at Urbana-Champaign, Urbana, IL 61801}
\email{cer2@illinois.edu, yadav7@illinois.edu, ayong@illinois.edu}
\begin{document}
\pagestyle{plain}

\mbox{}

\date{August 29, 2019}

\begin{abstract}
This chapter concerns edge labeled Young tableaux, introduced by H.~Thomas and the third author. It is
used to model equivariant Schubert calculus of Grassmannians.
We survey results, problems, conjectures, together with their influences from combinatorics, algebraic and symplectic geometry, linear algebra,
and computational complexity. We report on a new shifted analogue of edge labeled tableaux. Conjecturally, this gives 
a Littlewood-Richardson rule for the structure constants of the D.~Anderson-W.~Fulton ring, which is related to the equivariant cohomology of isotropic Grassmannians. 
\end{abstract}

\maketitle

\begin{center}
\emph{To William Fulton on his eightieth birthday, for inspiring generations.} 
\end{center}

\section{Introduction}\label{sec:intro}

\subsection{Purpose}\label{sec:purpose}
Singular cohomology is a functor between the categories
\[\{\text{topological spaces},\text{continuous maps}\} \to 
\{\text{graded algebras}, \text{homomorphisms}\}.\]
The cohomology functor links the geometry of Grassmannians to symmetric functions and Young tableaux. However,
this does not take into account the large torus action on the Grassmannian.
A similar functor for topological spaces with continuous group actions is equivariant cohomology. 

\begin{quotation}
\emph{What are equivariant analogues for these centerpieces of algebraic combinatorics?}
\end{quotation}

We posit a comprehensive answer, with applications, and future perspectives. 

\subsection{Schubert calculus}\label{sec:1.2} Let $X={\sf Gr}_k({\mathbb C}^n)$ be the Grassmannian of $k$-dimensional planes in ${\mathbb C}^n$. The group ${\sf GL}_n$
of invertible $n\times n$ matrices acts transitively on $X$ by change of basis. Let ${\sf B}_{-}\subset {\sf GL}_n$ be a opposite Borel subgroup of lower triangular
matrices. ${\sf B}_{-}$ acts on $X$ with finitely many orbits $X_{\lambda}^{\circ}$ where
$\lambda$ is a partition (identified with its Young diagram, in English notation) that is contained in the $k\times (n-k)$ rectangle $\Lambda$.
These \emph{Schubert cells} satisfy  $X_{\lambda}^{\circ}\cong {\mathbb C}^{k(n-k)-|\lambda|}$ where
$|\lambda|=\sum_i {\lambda_i}$. Their closures, the \emph{Schubert varieties}, satisfy 
\[X_{\lambda}:=\overline{X_{\lambda}^{\circ}}=\coprod_{\mu\supseteq \lambda}X_{\mu}^{\circ}.\]
Let $\nu^{\vee}$ be the $180^\circ$-rotation of $k\times (n-k)\setminus \nu$.
Suppose $|\lambda|+|\mu|+|\nu^{\vee}|=k(n-k)=\dim X$. By Kleiman transversality \cite{Kleiman}, there is a dense open 
${\mathcal O}\subset {\sf GL}_n\times {\sf GL}_n \times {\sf GL}_n$ such that
\[c_{\lambda,\mu}^{\nu}:=\#\{g_1 \cdot X_{\lambda} \cap g_2 \cdot X_{\mu}\cap g_3\cdot X_{\nu^{\vee}}\}\in {\mathbb Z}_{\geq 0}\]
is independent of $(g_1,g_2,g_3)\in {\mathcal O}$. Each $c_{\lambda,\mu}^{\nu}$ is called a \emph{Littlewood-Richardson coefficient}. Modern Schubert calculus is concerned with these coefficients, as well as their generalizations/analogues (from varying the space $X$ or cohomology theory).

Let $\sigma_{\lambda}\in H^{2|\lambda|}(X)$ be the Poincar\'e dual to $X_{\lambda}$. These \emph{Schubert
classes} form a ${\mathbb Z}$-linear basis of $H^{*}(X)$ and
\[\sigma_{\lambda}\smallsmile \sigma_{\mu}=\sum_{\nu\subseteq \Lambda} c_{\lambda,\mu}^{\nu}\sigma_{\nu}.\]

The \emph{Schur function} $s_{\lambda}$ is the generating
series
$s_{\lambda}=\sum_{T}x^T$
for semistandard Young tableaux of shape $\lambda$, \emph{i.e.}, row weakly increasing and column strictly increasing fillings of $\lambda$ with elements of ${\mathbb N}$.
The weight of $T$ is $x^T:=\prod_i x_i^{\#i\in T}$. For example, if $\lambda=(2,1)$, the semistandard tableaux
are
\[\tableauS{1&1\\2} \ \ \ \ \tableauS{1&2\\2} \ \ \ \  \tableauS{1&3\\2} \ \ \ \   \tableauS{1&1\\3} \ \ \ \  \tableauS{1&2\\3}\ \ \ \  \tableauS{1&3\\3}\ \ \ \  \tableauS{2&2\\3}\ \ \ \  \tableauS{2&3\\3} \ \ \ \ \cdots \] 
Hence $s_{(2,1)}=x_1^2 x_2 +x_1 x_2^2 +x_1 x_2 x_3 +x_1^2 x_3 +x_1 x_2 x_3 +x_1 x_3^2 +x_2^2 x_3 +x_2 x_3^2+\cdots$. 
Schur functions form a ${\mathbb Z}$-linear basis of ${\sf Sym}$, the ring of symmetric functions in infinitely many variables. 

The map $\sigma_{\lambda}\mapsto s_{\lambda}$ induces a ring isomorphism 
\[H^{*}(X)\cong {\sf Sym}/I\]
where $I$ is the ideal $\langle s_{\lambda}:\lambda\not\subseteq \Lambda\rangle$. Therefore in ${\sf Sym}$,
\begin{equation}
\label{eqn:Schurprod}
s_{\lambda}\cdot s_{\mu}=\sum_{\nu} c_{\lambda,\mu}^{\nu} s_{\nu}.
\end{equation}
To compute Schubert calculus of $X$, it suffices to determine (\ref{eqn:Schurprod}) by working with Schur \emph{polynomials} in only finitely many variables $x_1,\ldots,x_k$. Better yet,
the \emph{jeu de taquin} theory of Young tableaux, introduced by M.-P.~Sch\"utzenberger \cite{Schutzenberger} gives a combinatorial rule for
computing $c_{\lambda,\mu}^{\nu}$; this is summarized in Section~\ref{sec:Young}. 
What we have discussed thus far constitutes textbook material on Schubert calculus; see, \emph{e.g.}, W.~Fulton's \cite{Fulton:YT}.

\subsection{Overview} This chapter describes an equivariant analogue of M.-P.~Sch\"utzenberger's theory, due to
 H.~Thomas and the third author; in short, one replaces Young tableaux with \emph{edge labeled tableaux}. Now, we hasten to offer an \emph{apologia}: such tableaux are not the only combinatorial model to compute equivariant Schubert calculus. For example, one has work of A.~Molev-B.~Sagan \cite{Molev.Sagan:99} and
the \emph{puzzles} of A.~Knutson-T.~Tao \cite{Knutson.Tao}. The latter  has had important recent followup, see, \emph{e.g.}, A.~Knutson-P.~Zinn-Justin's \cite{KnutsonZinn} and the references
therein. One also has the tableaux of V.~Kreiman \cite{Kreiman} or A.~Molev \cite{Molev}. 

That said, we wish to argue how the edge labeled model is a handy and flexible viewpoint. It
has been applied to obtain equivariant analogues of a number of theorems (delineated in Section~\ref{sec:Young} and~\ref{sec:Horn}). 
Another application, due to O.~Pechenik and the third author \cite{Pechenik.Yong}, is to Schubert calculus for the 
equivariant $K$-theory of $X$. Translating the combinatorics into puzzle language allowed for a proof (of a correction) of a conjecture of A.~Knutson-R.~Vakil about the same structure constants \cite{Pechenik.Yong:compositio}. However, as we wish to restrict to equivariant cohomology proper, this direction is not part of our discussion.

There is an important frontier to cross, that is, 
the still unsolved problem of finding a combinatorial rule for equivariant Schubert calculus of maximal orthogonal and
Lagrangian Grassmannians. The non-equivariant story is explained in Section~\ref{sec:PQstuff}. We explain the problem in Section~\ref{sec:YZ} together with some recent developments of C.~Monical \cite{Monical} and of the authors \cite{RYY}. The latter work shows that the combinatorial problems concerning the two spaces are essentially equivalent. 

This brings us to the principal new announcement of this work (Section~\ref{sec:shiftededge}):
the notion of \emph{shifted edge labeled tableaux}. We define an analogue of \emph{jeu de taquin} and use this to conjecturally define an associative ring (which we prove to also be commutative). The introduction of this ring is stimulated by recent work of D.~Anderson-W.~Fulton (see Section~\ref{sec:AndersonFultonstuff}) who define a ring connected to the equivariant cohomology of Lagrangian Grassmannians. Conjecturally, the two rings are isomorphic. This provides our strongest evidence to date of the applicability
of the edge labeled approach to the aforementioned open problem; we know of no similar results using other combinatorial models.

As this work is partially expository and partly an announcement, we limited the number of complete proofs in order to keep the focus on the
high-level research objectives. Where possible, we have sketched arguments (with references) and indicated those results which may be taken as
an exercise for the interested reader. These exercises are warmups for the conjectures and open problems contained herein.

\section{Equivariant cohomology of Grassmannians}\label{sec:Equiv}

\subsection{Generalities}\label{sec:Equivgen} We recall some general notions about equivariant cohomology. References that we consulted are L.~Tu's synopsis \cite{Tu}, J.~Tymoczko's exposition \cite{Tymoczko.eq} and A.~Knutson-T.~Tao's \cite[Section~2]{Knutson.Tao}.

Let ${\mathcal M}$ be a topological space with the continuous action of a topological group $\G$. If $\G$ acts freely on 
${\mathcal M}$, then in fact the equivariant cohomology ring $H_{\G}^{*}({\mathcal M})$ is $H^{*}({\mathcal M}/\G)$, see \cite[Proposition 2.1]{Tymoczko.eq}.
However, in general the action is not free, and ${\mathcal M}/\G$ might be, \emph{e.g.}, non-Hausdorff.
 Borel's \emph{mixing space construction} introduces a contractible space
${\sf EG}$ on which $\G$ acts freely. Thus $\G$'s diagonal action on ${\sf EG}\times {\mathcal M}$ is free and 
\[H_{\G}^{*}({\mathcal M}):=H^{*}({\sf EG}\times {\mathcal M}/\G).\] 
The space ${\sf EG}$ is the total space
of the \emph{universal principle $\G$-bundle} $\pi:{\sf EG}\to {\sf BG}$ where ${\sf BG}={\sf EG}/{\sf G}$ is the classifying space of ${\sf G}$. Here, universality means that if $\rho:P\to {\mathcal M}$ is any $\G$-bundle, there exists a unique map
$f:{\mathcal M}\to {\sf BG}$ (up to homotopy) such that $P\cong f^*({\sf EG})$. By functoriality, the constant map $c:{\mathcal M}\to \{pt\}$ induces a homomorphism
$c^{*}:H^*_{\G}(pt)\to H^*_{\G}({\mathcal M})$. Hence $H^{*}_{\G}({\mathcal M})$ is a module over $H^{*}_{\G}(pt)$ by $\beta\cdot \kappa:=c^*(\beta)\kappa$ for
$\beta\in H^{*}_{\G}(pt)$ and $\kappa\in H^{*}_{\G}({\mathcal M})$. While ordinary (singular) cohomology of a point is $\mathbb Z$, 
$H_{\G}^{*}(pt)$ is \emph{big}. For instance, if $\G={\sf T}$ is an $n$ torus $(S^1)^n$, then $H_{{\sf T}}^{*}(pt)={\mathbb Z}[t_1,\ldots,t_n]$.
 
If we presume $\G$ is a algebraic group acting on a smooth algebraic variety $M$, 
these notions have versions in the algebraic category; see, \emph{e.g.}, D.~Anderson's \cite{Dave.eq}.

\subsection{The Grassmannian} 
Concretely, if $\lambda=(\lambda_1\geq \lambda_2\geq \ldots \geq \lambda_k\geq 0)$ then 
\begin{equation}
\label{eqn:concrete}
X_{\lambda}=\{V\in X| \dim(V\cap F^{n-k+i-\lambda_i})\geq i, 1\leq i\leq k\},
\end{equation}
where $F^d={\rm span}(e_n,e_{n-1},\ldots, e_{n-d+1})$ and $e_i$ is the $i$-th standard basis vector; see \cite[Section~9.4]{Fulton} for details.

Let ${\sf T}\subset {\sf GL}_n$ be the torus of invertible diagonal matrices. Then from (\ref{eqn:concrete}), $X_{\lambda}$
is  ${\sf T}$-stable. Therefore, $X_{\lambda}$ admits an equivariant Schubert class $\xi_{\lambda}$
in the ${\sf T}$-equivariant
cohomology ring $H^{*}_{\sf T}(X)$. By what we have recounted in Section~\ref{sec:Equivgen}, $H^{*}_{\sf T}(X)$ is a module over 
\begin{equation}
\label{eqn:HT}
H^{*}_{\sf T}(pt):={\mathbb Z}[t_1,t_2,\ldots,t_n].
\end{equation}
The equivariant
Schubert classes form a $H_{\sf T}^*(pt)$-module basis of $H_{\sf T}^*(X)$. Therefore,
\begin{equation}
\label{eqn:HTexp}
\xi_{\lambda}\cdot \xi_{\mu} =\sum_{\nu\subseteq \Lambda} C_{\lambda,\mu}^{\nu} \xi_{\nu},
\end{equation}
where $C_{\lambda,\mu}^{\nu}\in H_{\sf T}^{*}(pt)$. For more details about equivariant cohomology specific to flag varieties we point the reader to 
\cite{Dave.eq} 
and S.~Kumar's textbook \cite[Chapter~XI]{Kumar}.

Let $\beta_i:=t_i-t_{i+1}$. D.~Peterson conjectured, and W.~Graham \cite{Graham} proved\begin{theorem}[Equivariant positivity \cite{Graham}]\label{thm:Graham}
$C_{\lambda,\mu}^{\nu}\in {\mathbb Z}_{\geq 0}[\beta_1,\ldots,\beta_{n-1}]$.\footnote{Actually, W.~Graham proved that for any
generalized flag variety ${\sf H}/{\sf B}$, the equivariant Schubert structure constant is expressible as a nonnegative integer polynomial in the simple roots of the (complex, semisimple) Lie group ${\sf H}$.}

\end{theorem}

In fact, $\deg C_{\lambda,\mu}^{\nu} = |\lambda|+|\mu|-|\nu|$
and  $C_{\lambda,\mu}^{\nu}=0$ 
unless $|\lambda|+|\mu|\geq |\nu|$.  In the case $|\lambda|+|\mu|=|\nu|$,
$C_{\lambda,\mu}^{\nu}=c_{\lambda,\mu}^{\nu}$ is the Littlewood-Richardson coefficient of Section~\ref{sec:intro}.

Fix
a grid with $n$ rows and $m \geq n + \lambda_1 - 1$ columns. The \emph{initial diagram} places $\lambda$ in the northwest corner of this grid.
For example,
if $\lambda=(3,2,0,0)$, the
initial diagram for $\lambda$ is the first
of the three below.
\[\left[\begin{matrix}
+ & + & + & \cdot & \cdot & \cdot\\
+ & + & \cdot & \cdot & \cdot & \cdot \\
\cdot & \cdot & \cdot & \cdot & \cdot & \cdot \\
\cdot & \cdot & \cdot & \cdot & \cdot & \cdot
\end{matrix}\right] \ \ \ 
\left[\begin{matrix}
+ & + & \cdot & \cdot & \cdot & \cdot\\
+ & \cdot & \cdot & + & \cdot & \cdot \\
\cdot & \cdot & \cdot & \cdot & \cdot & \cdot \\
\cdot & \cdot & \cdot & + & \cdot & \cdot
\end{matrix}\right] \ \ \ 
\left[\begin{matrix}
+ & \cdot & \cdot & \cdot & \cdot & \cdot\\
\cdot & \cdot & + & \cdot & \cdot & \cdot \\
\cdot & + & \cdot & \cdot & + & \cdot \\
\cdot & \cdot & \cdot & + & \cdot & \cdot
\end{matrix}\right]
\]
A \emph{local move} is a change of any $2\times 2$ subsquare of the form
\[
\begin{matrix}
+ & \cdot\\
\cdot & \cdot
\end{matrix}\ \ \  \ \ 
\mapsto  \ \ \ \ \ 
\begin{matrix}
\cdot & \cdot\\
\cdot & +
\end{matrix}\]
A \emph{plus diagram} is any configuration of $+$'s in the grid resulting from some number of local moves starting from the initial diagram for $\lambda$. We have given two more examples of plus diagrams for $\lambda=(3,2,0,0)$.

Let ${\sf Plus}(\lambda)$ denote the set of plus diagrams for $\lambda$. If $P\in {\sf Plus}(\lambda)$, let 
\[{\sf wt}_x(P)=x_1^{\alpha_1} x_2^{\alpha_2} \cdots x_n^{\alpha_n}.\] 
Here, $\alpha_i$ is the number of $+$'s in the $i$th row of $P$. 
For instance, if $P$ is the rightmost diagram shown above, ${\sf wt}_x(P)=x_1 x_2 x_3^2 x_4$. 
A more refined statistic is 
\[{\sf wt }_{x,y}(P)=\prod_{(i,j)}x_i-y_j.\] 
The product is over those $(i,j)$ with a $+$ in row $i$ and column $j$ of $P$. For the same $P$,
\[{\sf wt}_{x,y}(P)= 
(x_1-y_1)(x_2-y_3)(x_3-y_2)(x_3-y_5)(x_4-y_4).\]
Let $X=\{x_1,x_2,\ldots,x_n\}$ and $Y=\{y_1,y_2,\ldots, y_{n + \lambda_1 - 1}\}$ be two collections of indeterminates. The \emph{factorial Schur function} is
\[s_{\lambda}(X; Y)=\sum_{P\in {\sf Plus}(\lambda)}  {\sf wt}_{x,y}(P).\]
This description arises in, \emph{e.g.}, \cite{KMY}. Moreover, it is an exercise to show 
\[s_{\lambda}(X)=\sum_{P\in {\sf Plus}(\lambda)}  {\sf wt}_x(P)=s_{\lambda}(X;0,0,\ldots).\]

The factorial Schur polynomials  form a 
${\mathbb Z}[Y]$-linear basis of ${\sf Sym} \otimes_{\mathbb Q} {\mathbb Z}[Y]$. In addition,
\begin{equation}
\label{eqn:factorialprod}
s_{\lambda}(X; Y)s_{\mu}(X; Y)=\sum_{\nu}C_{\lambda,\mu}^{\nu} \ s_{\nu}(X; Y).
\end{equation}
For example, one checks that 
\[s_{(1,0)}(x_1,x_2;Y)^2=s_{(2,0)}(x_1,x_2;Y)
+s_{(1,1)}(x_1,x_2;Y)+(y_3-y_2)s_{(1,0)}(x_1,x_2;Y).\]

In view of (\ref{eqn:HT}), the definition of $C_{\lambda,\mu}^{\nu}$ in terms of (\ref{eqn:factorialprod}) gives a definition of $H_{\sf T}^{*}({\sf X})$ that suffices
for our combinatorial ends. 

\subsection{Equivariant restriction}\label{subsection:GKM} We may further assume $\G$ is an algebraic
$n$-torus ${\sf T}$ which acts on ${\mathcal M}$ with finitely many isolated fixed points ${\mathcal M}^{\sf T}$. A feature of equivariant cohomology is that the inclusion ${\mathcal M}^{\sf T}$ into ${\mathcal M}$
induces an module \emph{injection} 
\[H_{\sf T}^{*}({\mathcal M})\hookrightarrow H_{\sf T}^{*}({\mathcal M}^{\sf T})\cong \bigoplus_{{\mathcal M}^{\sf T}}H_{{\sf T}}^{*}(pt)\cong \bigoplus_{{\mathcal M}^{\sf T}}{\mathbb Z}[t_1,\ldots,t_n].\] 
For each ${\sf T}$-invariant cycle $Y$ in ${\mathcal M}$, one has an equivariant cohomology class $[Y]_{\sf T}\in H_{\sf T}^{*}({\mathcal M})$. This injection says that this class is a $\#\{{\mathcal M}^{\sf T}\}$-tuple of polynomials $[Y]|_x$ where $x\in {\mathcal M}^{\sf T}$.  Each polynomial $[Y]|_x$ is an \emph{equivariant restriction}. Under certain assumptions on ${\mathcal M}$, which cover all generalized flag manifolds (such as Grassmannians), one has a divisibility condition on the restrictions. This alludes to the
general and influential package of ideas contained in Goresky-Kottwitz-MacPherson (``GKM'') theory; we refer to \cite{GKM} as well as J.~Tymoczko's survey \cite{Tymoczko.eq}.

One has more precise results (predating \cite{GKM}) for any generalized flag variety.  Work of B.~Kostant-S.~Kumar \cite{KostantKumar} combined with a formula of
H.~Anderson-J.~Jantzen-W.~Soergel \cite{AJS} describes $[Y]|_x$ where $Y$ is a Schubert variety and $x$ is one of the ${\sf T}$-fixed points. For restriction formulas
specific to the Grassmannian $M={\sf Gr}_{k}({\mathbb C}^n)$, see the formula of T.~Ikeda-H.~Naruse \cite[Section~3]{Ikeda.Naruse} in terms of \emph{excited Young diagrams}.\footnote{One can give another formula in terms of certain
specializations of the factorial Schur polynomial; see, \emph{e.g.}, \cite[Theorem~5.4]{Ikeda.Naruse} and the associated references.} From either formula, one 
sees immediately that
\begin{equation}
\label{eqn:subsetabc124}
\xi_{\lambda}|_{\mu}=0 \text{\ unless $\lambda\subseteq \mu$},
\end{equation}
and
\begin{equation}
\label{eqn:gghe}
\xi_{\mu}|_{\mu}\neq 0.
\end{equation}

The central difference between this picture of equivariant cohomology
of Grassmannians and the Borel-type presentation (due to A.~Arabia \cite{Arabia}) is that multiplication can be done \emph{without relations} and computed by
pointwise multiplication of the restriction polynomials. In particular, each instance of (\ref{eqn:HTexp}) gives rise to $\binom{n}{k}$ many
polynomial identities. For example, combining (\ref{eqn:HTexp}) and (\ref{eqn:subsetabc124}) gives
\[\xi_{\lambda}|_{\mu}\cdot \xi_{\mu}|_{\mu}=C_{\lambda,\mu}^{\mu}\xi_{\mu}|_{\mu}.\]
By (\ref{eqn:gghe}), this implies
\begin{equation}
\label{eqn:Arabia}
\xi_{\lambda}|_{\mu}=C_{\lambda,\mu}^{\mu},
\end{equation}
which is a fact first noted (for generalized flag varieties) by A.~Arabia \cite{Arabia}.
It follows that:
\begin{equation}
\label{eqn:equivpieriabc}
\xi_{\lambda}\cdot \xi_{(1)}=\xi_{\lambda}|_{(1)} \xi_{\lambda}+\sum_{\lambda^+} \xi_{\lambda^+},
\end{equation}
where $\lambda^+$ is obtained by adding a box to $\lambda$; see, \emph{e.g.}, \cite[Proposition~2]{Knutson.Tao}. Alternatively, it is an exercise to derive it from a vast generalization due to C.~Lenart-A.~Postnikov's \cite[Corollary~1.2]{Lenart.Postnikov}.

Since $H_{\sf T}^{*}(X)$ is an associative ring, one has
\[ (\xi_{\lambda} \cdot \xi_{\mu}) \cdot \xi_{(1)}= \xi_{\lambda}\cdot ( \xi_{\mu}\cdot \xi_{(1)}),\]
which when expanded using (\ref{eqn:HTexp}) and (\ref{eqn:equivpieriabc}) gives a recurrence that uniquely determines the structure coefficients; we call this
the \emph{associativity recurrence}. Since we will not explicitly need it 
in this chapter we leave it as an exercise (see \cite[Lemma~3.3]{Thomas.Yong} and the references therein).\footnote{We give an analogue (\ref{eqn:Brec}) of the associativity
recurrence in our proof of
Theorem~\ref{thm:typeBCpower2}.}

\section{Young tableaux and jeu de taquin}\label{sec:Young}

There are several combinatorial rules for the Littlewood-Richardson coefficient $c_{\lambda,\mu}^{\nu}$. The one whose theme will pervade this chapter is the \emph{jeu de taquin} rule, which, moreover is the first \emph{proved} rule for the coefficients \cite{Schutzenberger}. 

Let $\nu/\lambda$ be a skew shape. A \emph{standard tableau} $T$ of shape $\nu/\lambda$ is a bijective filling of $\nu/\lambda$ with $1,2,\ldots,|\nu/\lambda|$
such that the rows and columns are increasing. Let ${\sf SYT}(\nu/\lambda)$ be the set of all such tableaux.
An \emph{inner corner} ${\sf c}$ of $\lambda/\mu$ is a maximally southeast box of $\mu$. For $T \in {\sf SYT}(\lambda/\mu)$, a 
\emph{jeu de taquin slide} ${\sf jdt}_{\sf c}(T)$
is obtained as follows. Initially place $\bullet$ in ${\sf c}$, and apply one of the following \emph{slides}, according how $T$ looks near ${\sf c}$:
\begin{itemize}
\item[(J1)] $\tableauS{\bullet & a\\ b}\mapsto \tableauS{b & a\\ \bullet }$ (if $b<a$, or $a$ does not exist)
\smallskip
\item[(J2)] $\tableauS{\bullet & a\\ b}\mapsto \tableauS{a & \bullet\\ b}$ (if $a<b$, or $b$ does not exist)
\end{itemize}
Repeat application of (J1) or (J2) on the new box ${\sf c}'$ where $\bullet$ arrives at. End when  $\bullet$ arrives at a box ${\sf d}$ 
of $\lambda$ that has no labels
south or east of it. Then ${\sf jdt}_{\sf c}(T)$ is obtained by erasing $\bullet$.

A \emph{rectification} of $T\in {\sf SYT}(\lambda/\mu)$ is defined iteratively. Pick an inner corner ${\sf c}_0$ of $\lambda/\mu$ and compute
$T_1:={\sf jdt}_{{\sf c}_0}(T) \in {\sf SYT}(\lambda^{(1)}/\mu^{(1)})$. Let ${\sf c}_1$ be an inner corner of $\lambda^{(1)}/\mu^{(1)}$ and compute
$T_2:={\sf jdt}_{{\sf c}_1}(T_1)\in {\sf SYT}(\lambda^{(2)}/\mu^{(2)})$. Repeat $|\mu|$ times, arriving at a standard tableau of straight (\emph{i.e.}, partition) shape. Let ${\sf Rect}_{\{{\sf c}_i\}}(T)$ be the result.

\begin{theorem}[First fundamental theorem of jeu de taquin]\label{thm:firstjdt}
${\sf Rect}_{\{{\sf c}_i\}}(T)$ is independent of the choice of sequence of successive inner corners $\{{\sf c}_i\}$.
\end{theorem}

Theorem~\ref{thm:firstjdt} permits one to speak of \emph{the} rectification ${\sf Rect}(T)$.

\begin{example}\label{ex:rectWellDef} For instance, here are two different rectification orders for a tableau $T$. 
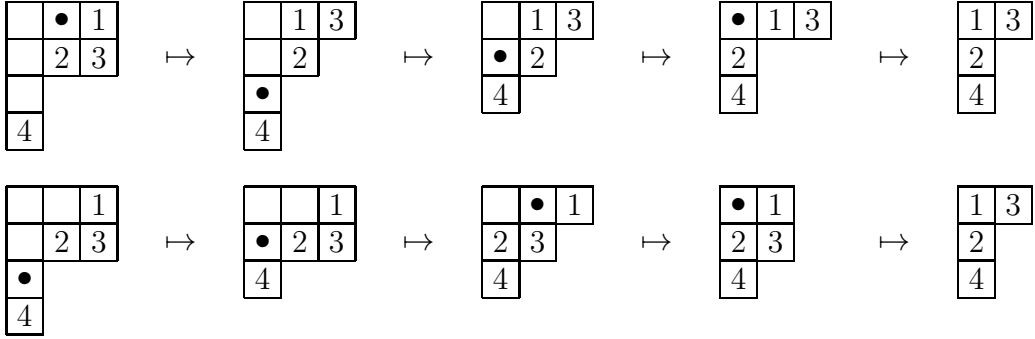
\[
\begin{picture}(400,120)
\put(0,105){$\tableauS{{\ }&{\bullet }&{ 1}\\{\ }&{2}&{3}\\{ \ }\\{4}}$}
\put(60,95){$\mapsto$}
\put(90,105){$\tableauS{{\ }&{1 }&{ 3}\\{\ }&{2}\\{\bullet }\\{4}}$}
\put(150,95){$\mapsto$}
\put(180,105){$\tableauS{{\ }&{1 }&{ 3}\\{\bullet }&{2}\\{4}\\}$}
\put(240,95){$\mapsto$}
\put(270,105){$\tableauS{{\bullet }&{1 }&{ 3}\\{2 }\\{4}}$}
\put(330,95){$\mapsto$}
\put(360,105){$\tableauS{{1 }&{3 }\\{2 }\\{4}}$}
\put(0,35){$\tableauS{{\ }&{\ }&{ 1}\\{\ }&{2}&{3}\\{\bullet }\\{4 }}$}
\put(60,25){$\mapsto$}
\put(90,35){$\tableauS{{\ }&{\ }&{ 1}\\{\bullet }&{2}&{3}\\{4 }}$}
\put(150,25){$\mapsto$}
\put(180,35){$\tableauS{{\ }&{\bullet }&{ 1}\\{2  }&{3}\\{4 }}$}
\put(240,25){$\mapsto$}
\put(270,35){$\tableauS{{\bullet }&{1 }\\{2  }&{3}\\{4 }}$}
\put(330,25){$\mapsto$}
\put(360,35){$\tableauS{{1 }&{3 }\\{2  }\\{4 }}$}
\end{picture}\]  
\end{example}

\begin{theorem}[Second fundamental theorem of jeu de taquin]\label{thm:secondjdt}
The cardinality 
\begin{equation}
\label{eqn:secondcard}
\#\{T\in {\sf SYT}(\nu/\lambda): {\sf Rect}(T)=U\}
\end{equation}
is independent of the choice of $U\in {\sf SYT}(\mu)$.
\end{theorem}

\begin{example}\label{ex:anyTabMu} Below are the tableaux $T\in{\sf SYT}((3,2,1)/(2,1))$ such that ${\sf Rect}(T)=U\in{\sf SYT}((2,1))$. Of the tableaux below, $T_1,T_2$ rectify to $U_1$ and $T_3,T_4$ rectify to $U_2$.
\begin{gather*}
T_1=\tableauS{{\ }&{\  }&{ 2}\\{\ }&{1}\\{3 }}
\qquad
T_2=\tableauS{{\ }&{\  }&{ 2}\\{\ }&{3}\\{1 }}
\qquad
T_3=\tableauS{{\ }&{\  }&{ 1}\\{\ }&{3}\\{2 }}
\qquad
T_4=\tableauS{{\ }&{\  }&{ 3}\\{\ }&{1}\\{2 }}
\\ 
U_1=\tableauS{{1} &{2} \\ {3}}
\qquad
U_2=\tableauS{{1} &{3} \\ {2}}
\end{gather*}
\end{example}

For proofs of Theorems~\ref{thm:firstjdt} and~\ref{thm:secondjdt} we recommend the self-contained argument found in M.~Haiman's \cite{Haiman}, which
is based on his theory of \emph{dual equivalence}.

\begin{theorem}[Jeu de taquin computes the Littlewood-Richardson coefficient]\label{thm:thirdjdt}
Fix $U\in {\sf SYT}(\mu)$. Then $c_{\lambda,\mu}^{\nu}$ equals the number (\ref{eqn:secondcard}).
\end{theorem}

\begin{example}\label{ex:LRcountTab} Continuing Example \ref{ex:anyTabMu}, fix $U:=U_1$. Then 
$c_{(2,1),(2,1)}^{(3,2,1)}=2=\#\{T_1,T_2\}$.\qed
\end{example}

It is convenient to fix a choice of tableau $U$ in Theorem~\ref{thm:thirdjdt}. Namely, let
$U=S_{\mu}$ be the superstandard tableau of shape $\mu$, which is obtained by filling the boxes of $\mu$ in English reading order
with $1,2,3,\ldots$. This is the choice made in Example~\ref{ex:LRcountTab}. A larger instance is
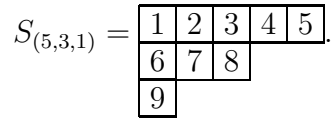
\[S_{(5,3,1)}=\tableauS{1&2&3&4&5\\ 6&7&8\\ 9}.\]

There are a number of ways to prove a Littlewood-Richardson rule such as Theorem~\ref{thm:thirdjdt}. We describe the two
that we will refer to in this chapter:

\noindent
$\bullet$ \emph{``Bijective argument'':}
In terms of (\ref{eqn:Schurprod}), the most direct is to establish a bijection between pairs $(A,B)$ of semistandard tableau of shape
$\lambda$ and $\mu$ respectively and pairs $(C,D)$ where $C\in {\sf SYT}(\nu/\lambda)$ such that ${\sf Rect}(C)=S_{\mu}$
and $D$ is a semistandard tableau of shape $\nu$. This can be achieved using the \emph{Robinson-Schensted correspondence}. 

\noindent
$\bullet$ \emph{``Associativity argument'':} 
This was used by A.~Knutson-T.~Tao-C.~Woodward \cite{KTW} and A.~Buch-A.~Kresch-H.~Tamvakis \cite{BKT}. Define a
 putative ring $(R,+,\star)$ with additive basis $\{[\lambda]:\lambda\subseteq \Lambda\}$
and product
\[[\lambda]\star[\mu]:=\sum_{\nu\subseteq \Lambda} {\overline c}_{\lambda,\mu}^{\nu}[\nu],\]
where ${\overline c}_{\lambda,\mu}^{\nu}$ is a collection of nonnegative integers.
Assume one can prove $\star$ is commutative and associative and moreover agrees with Pieri's rule, \emph{i.e.}, 
${\overline c}_{\lambda,(p)}^{\nu}=0$ unless $\nu/\lambda$ is a horizontal strip of size $p$, and equals $1$ otherwise. Then it follows that
$c_{\lambda,\mu}^{\nu}={\overline c}_{\lambda,\mu}^{\nu}$. One can apply this to prove Theorem~\ref{thm:thirdjdt}. The proofs of commutativity and associativity use Theorem~\ref{thm:firstjdt}. Typically, the hard step is the proof of associativity, which explains the nomenclature for this
proof technique. \qed

Another formulation of the Littlewood-Richardson rule is in terms of semistandard Young tableaux of shape
$\nu/\lambda$ and content $\mu$. These are fillings $T$ of $\nu/\lambda$ with $\mu_i$ many $i$'s, and such that the rows are
weakly increasing and columns and strictly increasing. The \emph{row reading word} is obtained by reading the entries of $T$ 
along rows, from right to left and from top to bottom. Such a word $(w_1,w_2,\ldots,w_{|\nu/\lambda|})$ 
is \emph{ballot} if for every fixed $i,k\geq 1$,
\[\#\{j\leq k: w_j=i\}\geq \#\{j\leq k: w_j=i+1\}.\]
A tableau is ballot if its reading word is ballot.
\begin{theorem}[Ballot version of the Littlewood-Richardson rule]\label{thm:LRsemi}
$c_{\lambda,\mu}^{\nu}$ equals the number of semistandard tableaux of shape $\nu/\lambda$ and content $\mu$ that are ballot.
\end{theorem}

\begin{example}\label{ex:LRballotTab}
Suppose $\lambda=(3,2,1),\mu=(3,2,1),\nu=(4,4,3,1)$. Below are the $3$ semistandard tableaux of shape $\nu/\lambda$ and content $\mu$ that are ballot.
\begin{gather*}
\tableauS{{\ }&{\  }&{\ }&{1}\\{\ }&{\ }&{1}&{2}\\{\ }&{1}&{2}\\{3}&}
\qquad
\tableauS{{\ }&{\  }&{\ }&{1}\\{\ }&{\ }&{1}&{2}\\{\ }&{1}&{3}\\{2}&}
\qquad
\tableauS{{\ }&{\  }&{\ }&{1}\\{\ }&{\ }&{1}&{2}\\{\ }&{2}&{3}\\{1}&}
\end{gather*}
\end{example}

\noindent
\emph{Proof sketch for Theorem~\ref{thm:LRsemi}:} Given a semistandard
tableau $T$, one creates a standard tableau $T'$ of the same shape by replacing all $\mu_1$ many $1$'s by $1,2,3,\ldots,\mu_1$ from left to right
and then replacing the (original) $\mu_2$ many $2$'s by $\mu_1+1,\mu_1+2,\ldots,\mu_1+\mu_2$ etc. This process is called
\emph{standardization}. Standardizing the tableaux in Example~\ref{ex:LRballotTab} respectively gives:
\begin{gather*}
\tableauS{{\ }&{\  }&{\ }&{3}\\{\ }&{\ }&{2}&{5}\\{\ }&{1}&{4}\\{6}&}
\qquad
\tableauS{{\ }&{\  }&{\ }&{3}\\{\ }&{\ }&{2}&{5}\\{\ }&{1}&{6}\\{4}&}
\qquad
\tableauS{{\ }&{\  }&{\ }&{3}\\{\ }&{\ }&{2}&{5}\\{\ }&{4}&{6}\\{1}&}
\end{gather*}

We claim standardization induces a bijection between the rules of Theorem~\ref{thm:LRsemi} and Theorem~\ref{thm:thirdjdt}. More precisely,
if $T$ is furthermore ballot, then $T'$ satisfies ${\sf rect}(T')=S_{\mu}$. We leave it as an exercise to establish this, \emph{e.g.}, by induction on $|\lambda|$.
\qed

There is a polytopal description of the Littlewood-Richardson rule derivable from Theorem~\ref{thm:LRsemi}. We first learned this
from a preprint version of \cite{Mulmuley}. Suppose $T$ is a semistandard tableaux of shape $\nu/\lambda$. 
Set
\[r_k^i=r_k^i(T)=\#\{\text{$k'$s in the $i$th row of $T$}\}.\]
Let $\ell(\mu)$ be the number of nonzero parts of $\mu$. By convention, let 
$r_{\ell(\mu)+1}^i =0,  r_k^{\ell(\nu)+1} = 0$. 

Now consider the following linear inequalities, constructed to describe the tableaux from Theorem~\ref{thm:LRsemi} that are counted by
$c_{\lambda,\mu}^{\nu}$.

\begin{itemize}
\item[(A)] Non-negativity: $r_{k}^{i}\geq0,  \ \ \forall i,k$.

\item[(B)] Shape constraints: $\lambda_i + \sum_{k}r_{k}^{i}=\nu_i, \ \ \forall i $.

\item[(C)] Content constraints: 
	$\sum_{i}r_{k}^{i}=\mu_k, \ \  \forall k.$ 

\item[(D)] Tableau constraints: $\lambda_{i+1} + \sum_{j\leq k}r_{j}^{i+1}\leq\lambda_{i} + \sum_{j'<k}r_{j'}^{i}, \ \ \forall i $.

\item[(E)] Ballot constraints: $\sum_{i'<i} r_k^{i'} \geq r_{k+1}^i + \sum_{i'<i}	r_{k+1}^{i'}, \ \ \forall i,k.$
\end{itemize}

Define a polytope
\[{\mathcal P}_{\lambda,\mu}^{\nu}=\{(r_k^i): \text{(A)--(E)}\}\subseteq {\mathbb R}^{\ell(\nu)\cdot\ell(\mu)}.\]

The following is a straightforward exercise, once one assumes Theorem~\ref{thm:LRsemi}:

\begin{theorem}[Polytopal Littlewood-Richardson rule]\label{thm:polytopalLR}
$c_{\lambda,\mu}^{\nu}$ counts the number of integer lattice points in ${\mathcal P}_{\lambda,\mu}^{\nu}$.
\end{theorem}

\begin{example}\label{ex:LRpolytope}
Using $\lambda,\mu,\nu$ as in Example \ref{ex:LRballotTab}, ${\mathcal P}_{\lambda,\mu}^{\nu}$ has $3$ integer lattice points 
$R(T_j)=(r(T_j)_{ik})=(r_k^i(T_j))$ below. Each $r_k^i(T_j)=\#\{\text{$k'$s in the $i$th row of $T_j$}\}$, where the $T_j$ are as in Example \ref{ex:LRballotTab}.
\[R(T_1)=\begin{pmatrix} 
1 & 0& 0 \\
1 & 1& 0 \\
1 & 1& 0 \\
0 & 0& 1 
\end{pmatrix}\qquad 
R(T_2)=\begin{pmatrix} 
1 & 0& 0 \\
1 & 1& 0 \\
1 & 0& 1 \\
0 & 1& 0 
\end{pmatrix}\qquad
R(T_3)=\begin{pmatrix} 
1 & 0& 0 \\
1 & 1& 0 \\
0 & 1& 1 \\
1 & 0& 0 
\end{pmatrix}
\]
\end{example}

Another conversation concerns ${\sf LR}_r=\{(\lambda,\mu,\nu): \ell(\lambda),\ell(\mu),\ell(\nu)\leq r: c_{\lambda,\mu}^{\nu}\neq 0\}$.

\begin{corollary}\label{cor:LRsemigroup}
${\sf LR}_r$ is a semigroup, \emph{i.e.}, if $(\lambda,\mu,\nu), (\alpha,\beta,\gamma)\in {\sf LR}_r$, then $(\lambda+\alpha,\mu+\beta,\nu+\gamma)\in {\sf LR}_r$.
\end{corollary}
\begin{proof}
Since $(\lambda,\mu,\nu), (\alpha,\beta,\gamma)\in {\sf LR}_r$, by Theorem~\ref{thm:polytopalLR} there exists lattice points
 $(r_{k}^i)\in {\mathcal P}_{\lambda,\mu}^{\nu}$ and $({\overline r}_{k}^i)\in {\mathcal P}_{\alpha,\beta}^{\gamma}$. By examination
 of the inequalities (A)-(E), clearly $(r_{i}^k+\overline{r}_i^k)$ is a lattice point in ${\mathcal P}_{\lambda+\alpha,\mu+\beta}^{\nu+\gamma}$, 
 and we are done by another application of Theorem~\ref{thm:polytopalLR}.
\end{proof}

This \emph{Littlewood-Richardson semigroup} ${\sf LR}_r$ is discussed in A.~Zelevinsky's article \cite{Zelevinsky}. That 
work concerns the
Horn and saturation conjectures (we will discuss these in Section~\ref{sec:Horn}). The point that a polytopal rule for $c_{\lambda,\mu}^{\nu}$
implies the semigroup property already appears in \emph{ibid.} It is also true that ${\sf LR}_r$ is \emph{finitely generated}. 
This is proved in A.~Elashvili's \cite{Elashvili}, who credits the argument to M.~Brion-F.~Knop from ``August-September, 1989''. This argument
(which applies more generally to tensor product multiplicities of any reductive Lie group) is not combinatorial. For another demonstration, see the proof 
of Proposition~\ref{prop:EqLRfinite}. Despite subsequent advances
in understanding Littlewood-Richardson coefficients, the following remains open:

\begin{problem}[\emph{cf.} Problems A and C of \cite{Zelevinsky}]\label{problem:zelev}
Explicitly give a finite (minimal) list of generators of ${\sf LR}_r$.
\end{problem}

In connection to the work of Section~\ref{sec:Horn}, a closely related problem has been solved by P.~Belkale
\cite{Belkale:rays}. His paper determines the extremal rays of the rational polyhedral cone defined by the points
of ${\sf LR}_r$. 

\section{Edge labeled tableaux and jeu de taquin}\label{sec:ejdt}

The history of the combinatorics of $C_{\lambda,\mu}^{\nu}$ is interesting in its own right.
The first combinatorial rule for $C_{\lambda,\mu}^{\nu}$ is due to A.~Molev-B.~Sagan \cite{Molev.Sagan:99}, who solved an even more
general problem. Under the obvious specialization, this rule is not  positive in the sense of
Theorem~\ref{thm:Graham}. The first such rule, in terms of \emph{puzzles}, was found by A.~Knutson and T.~Tao \cite{Knutson.Tao}. Subsequently,
visibly equivalent tableaux rules were independently found by V.~Kreiman \cite{Kreiman} and A.~Molev \cite{Molev}. Later, P.~Zinn-Justin \cite{ZinnJustin}
studied the puzzle rule of \cite{Knutson.Tao} based on the quantum integrability of the tiling model that underlies puzzles. Coming full circle,
a point
made in \cite[Section 6.5]{ZinnJustin} is that the rule of \cite{Molev.Sagan:99} \emph{does} provide a positive rule after all, under the ``curious
identity'' of \cite[Section 6.4]{ZinnJustin}.

This work takes a different view. It is about the \emph{edge labeled tableaux} introduced by H.~Thomas and the third author \cite{Thomas.Yong}.
A \emph{horizontal edge} of $\nu/\lambda$ is an east-west line segment which either lies along the lower or
upper boundary of $\nu/\lambda$, or which separates two boxes of $\nu/\lambda$.
An \emph{equivariant filling} of $\nu/\lambda$ is an assignment of elements of $[N]:=\{1,2,\ldots,N\}$ to
the boxes of $\nu/\lambda$ or to a horizontal edge of $\nu/\lambda$. While every box contains a single label,
each horizontal edge holds an element of $2^{[N]}$.
An equivariant filling is \emph{standard} if every label in $[N]$ appears exactly once and moreover any label is:
\begin{itemize}
\item strictly smaller than any label in its southern edge and the label in the
box immediately below it;
\item strictly larger than any label in its northern edge and the label in the box
immediately above it; and
\item weakly smaller than the label in the box immediately to its right.
\end{itemize}
(No condition is placed on the labels of adjacent edges.) Let ${\sf EqSYT}(\nu/\lambda,\ell)$ be the set of equivariant standard tableaux with entries from
$[N]$.

\begin{example}
\label{exa:June25abc}
Let $\nu/\lambda=(4,3,2)/(3,2,1)\subseteq \Lambda=3\times 4$ and
\[\begin{picture}(100,50)
\put(0,35){$T=\tableauL{{\ }&{\ }&{ \ }&{5}\\{\ }&{ \ }&{4}\\{\ }&{6}}$}
\put(32,-7){$3$}
\put(63,31){$1,2$}
\end{picture}\]
Then $T\in {\sf EqSYT}(\nu/\lambda,6)$.\qed
\end{example}

Given an inner corner ${\sf c}$ and $T\in {\sf EqSYT}(\nu/\lambda,N)$, if none of the following possibilities
applies, terminate. Otherwise use the unique applicable case (below ${\sf c}$ contains the $\bullet$):
\begin{itemize}
\item[(J1)] $\tableauS{\bullet & a\\ b}\mapsto \tableauS{b & a\\ \bullet }$ (if $b<a$, or $a$ does not exist)
\medskip
\item[(J2)] $\tableauS{\bullet & a\\ b}\mapsto \tableauS{a & \bullet\\ b}$ (if $a<b$, or $b$ does not exist)
\item[(J3)] $\begin{picture}(40,30)\put(0,0){$\tableauS{\bullet & a}$}
\put(3,-5){$S$}
\end{picture} \mapsto
\begin{picture}(40,30)\put(0,0){$\tableauS{a & \bullet }$}
\put(3,-5){$S$}
\end{picture}$ (if $a<\min(S)$)

\item[(J4)] $\begin{picture}(40,30)\put(0,0){$\tableauS{\bullet & a}$}
\put(3,-5){$S$}
\end{picture} \mapsto
\begin{picture}(40,30)\put(0,0){$\tableauS{s & a }$}
\put(3,-5){$S'$}
\end{picture}$ (if $s:=\min(S)<a$ and $S':=S\setminus \{s\}$)
\end{itemize}
This \emph{equivariant jeu de taquin slide} into ${\sf c}$ is denoted by
${\sf Ejdt}_{{\sf c}}(T)$. Clearly,
${\sf Ejdt}_{{\sf c}}(T)$ is also a standard equivariant filling.

\emph{The} rectification of $T$, denoted ${\sf Erect}(T)$,
is the result of successively using ${\sf Ejdt}_{{\sf c}}$ by choosing ${\sf c}$ that is
eastmost among all choices of inner corners at each stage.

\begin{example}
\label{exa:1.1}
Continuing Example~\ref{exa:June25abc},
we use ``$\bullet$'' to indicate the boxes that are moved into during 
${\sf Erect}(T)$. The rectification of the third column is as follows:
\begin{equation}
\label{eqn:ex1.1.1}
\begin{picture}(150,50)
\put(0,35){$\tableauL{{\ }&{\ }&{ \bullet }&{5}\\{\ }&{ \ }&{4}\\{\ }&{6}}$}
\put(5,-7){$3$}
\put(39,31){$1,2$}
\put(110,35){$\tableauL{{\ }&{\ }&{ 1 }&{5}\\{\ }&{ \ }&{4}\\{\ }&{6}}$}
\put(115,-7){$3$}
\put(155,31){$2$}
\put(90,35){$\mapsto$}
\end{picture}
\end{equation}
The rectification of the second column given by:
\begin{equation}
\label{eqn:ex1.1.2}
\begin{picture}(400,50)
\put(0,35){$\tableauL{{\ }&{\ }&{ 1 }&{5}\\{\ }&{ \bullet }&{4}\\{\ }&{6}}$}
\put(5,-7){$3$}
\put(45,31){$2$}
\put(90,35){$\mapsto$}
\put(110,35){$\tableauL{{\ }&{\bullet }&{ 1 }&{5}\\{\ }&{ 4 }\\{\ }&{6}}$}
\put(115,-7){$3$}
\put(155,31){$2$}
\put(190,35){$\mapsto$}
\put(210,35){$\tableauL{{\ }&{1 }&{ \bullet }&{5}\\{\ }&{ 4 }\\{\ }&{6}}$}
\put(215,-7){$3$}
\put(255,31){$2$}
\put(290,35){$\mapsto$}
\put(310,35){$\tableauL{{\ }&{1 }&{ 2 }&{5}\\{\ }&{ 4 }\\{\ }&{6}}$}
\put(315,-7){$3$}
\end{picture}
\end{equation}
and finally the rectification of the first column given by:
\begin{equation}
\label{eqn:ex1.1.3}
\begin{picture}(450,55)
\put(5,35){$\tableauL{{\ }&{1 }&{ 2 }&{5}\\{\ }&{ 4 }\\{\bullet }&{6}}$}
\put(10,-7){$3$}
\put(85,25){$\mapsto$}
\put(100,35){$\tableauL{{\ }&{1 }&{ 2 }&{5}\\{\bullet }&{ 4 }\\{ 3 }&{6}}$}
\put(180,25){$\mapsto$}
\put(195,35){$\tableauL{{\ }&{1 }&{ 2 }&{5}\\{3 }&{ 4 }\\{ \bullet }&{6}}$}
\put(275,25){$\mapsto$}
\put(290,35){$\tableauL{{\bullet }&{1 }&{ 2 }&{5}\\{3 }&{ 4 }\\{ 6 }}$}
\put(350,25){$\mapsto\cdots\mapsto$}
\put(400,35){$\tableauL{{1 }&{2 }&{ 5 }\\{3 }&{ 4 }\\{ 6 }}$}
\end{picture}
\end{equation}
Here the ``$\mapsto\cdots\mapsto$'' refers to slides moving the
$\bullet$ right in the first row.\qed
\end{example}

We now define ${\sf Ejdtwt}(T)\in {\mathbb Z}[t_1,\ldots,t_n]$
for a standard tableau $T$.
Each box $x$ in $\Lambda$ has a (Manhattan) \emph{distance} from the lower-left box: suppose $x$ has matrix coordinates 
$(i,j)$, then
\[\dist(x) := k+j-i.\]
Next, assign $x\in\Lambda$ the weight $\beta(x)=t_{\dist(x)}-t_{\dist(x)+1}$. 

\emph{If after rectification of a column, the label ${\mathfrak l}$
still remains an edge label, ${\sf Ejdtfactor}({\mathfrak l})$ is declared to be zero.} Otherwise, suppose
an edge label ${\mathfrak l}$ \emph{passes} through a box
$x$ if it occupies $x$ during the equivariant rectification of the column
of $T$ in which
${\mathfrak l}$ begins. Let the boxes passed be
$x_1,x_2,\ldots,x_s$. Also, when the rectification of a column is
complete, suppose the filled boxes strictly to the right of $x_s$ are
$y_1,\ldots,y_t$. Set
\[{\sf Ejdtfactor}({\mathfrak l})=(\beta(x_1)+\beta(x_2)+\cdots+\beta(x_s))+
(\beta(y_1)+\beta(y_2)+\cdots+\beta(y_t)).\]
Notice that since the boxes $x_1,\ldots,x_s, y_1,\ldots, y_t$
form a hook inside $\nu$, ${\sf Ejdtfactor}(i)=t_e-t_f$ with $e<f$. Now define
\[{\sf Ejdtwt}(T):=\prod_{{\mathfrak l}} {\sf Ejdtfactor}({\mathfrak l}),\]
where the product is over all edge labels ${\mathfrak l}$ of $T$.

\begin{theorem}[Edge labeled jeu de taquin rule \cite{Thomas.Yong}]
\label{thm:mainequiv}
\[C_{\lambda,\mu}^{\nu}=\sum_T {\sf Ejdtwt}(T),\]
where the sum is over all $T\in {\sf EqSYT}(\nu/\lambda,|\mu|)$ such that
${\sf Erect}(T)=S_{\mu}$.
\end{theorem}

Since each ${\sf Ejdtfactor}({\mathfrak l})$ is a positive sum of the
indeterminates $\beta_i=t_{i}-t_{i+1}$, Theorem~\ref{thm:mainequiv}
expresses $C_{\lambda,\mu}^{\nu}$ as a polynomial with positive
coefficients in the $\beta_i$'s, in agreement with Theorem~\ref{thm:Graham}.

\begin{example}\label{ex:eqLRcountTab}
Below each $x\in\Lambda=2\times 3$ is filled with $\beta(x)$: 
\[\ytableausetup
{boxsize=2em}\begin{ytableau}
 \scriptscriptstyle t_2-t_3 & \scriptscriptstyle t_3-t_4 & \scriptscriptstyle t_4-t_5 \\
 \scriptscriptstyle t_1-t_2  & \scriptscriptstyle t_2-t_3 & \scriptscriptstyle t_3-t_4
\end{ytableau}\]

$T_1,T_2$ below are those  $T\in {\sf EqSYT}((3,3)/(2,2),3)$ such that
${\sf Erect}(T)=S_{(2,1)}$ with nonzero weight: \[
\begin{picture}(350,50)
\put(0,25){$T_1=\tableauL{{\ }&{\ }&{ 2 }\\{\ }&{ \ }&{3}}$}
\put(33,3){$1$}
\put(125,25){$T_2=\tableauL{{\ }&{\ }&{ 2 }\\{\ }&{ \ }&{3}}$}
\put(180,3){$1$}
\put(250,25){$T_3=\tableauL{{\ }&{\ }&{ 1 }\\{\ }&{ \ }&{3}}$}
\put(322,21){$2$}
\end{picture}
\]
While $T_3$ also satisfies
${\sf Erect}(T)=S_{(2,1)}$, it has weight zero since there is an edge label when one
has completed rectifying the third column.

Theorem \ref{thm:mainequiv} asserts 
\begin{align*}
    C_{(2,2),(2,1)}^{(3,3)}&={\sf Ejdtwt}(T_1)+{\sf Ejdtwt}(T_2)\\
    &=
[((t_1-t_2)+(t_2-t_3))+(t_3-t_4)]+[((t_2-t_3)+(t_3-t_4))+(t_4-t_5)]\\
&=(t_2-t_5)+(t_1-t_4).
\end{align*}
\end{example}

In order to prove Theorem~\ref{thm:mainequiv}, one wishes to adapt the general strategy indicated at the end of Section~\ref{subsection:GKM}. 
If we define ${\overline C}_{\lambda,\mu}^{\nu}:=\sum_T {\sf Ejdtwt}(T)$ (as in Theorem~\ref{thm:mainequiv}) and show this collection of numbers
satisfies the associativity recurrence, one shows $C_{\lambda,\mu}^{\nu}={\overline C}_{\lambda,\mu}^{\nu}$, as desired.\footnote{It would be quite interesting
to give a ``bijective argument'' (in the sense of Section~\ref{sec:Young}) using the combinatorial description of the factorial Schur polynomials.}

The rule of Theorem~\ref{thm:mainequiv} appears to be too ``rigid'' to carry out this strategy. Instead, in \cite{Thomas.Yong}, a more ``flexible
version'' in terms of semistandard edge labeled tableaux is introduced, together with a corresponding collection of jeu de taquin slides.
While we do not wish to revisit the rather technical list of slide rules here, one of the consequences is a ballot rule for $C_{\lambda,\mu}^{\nu}$,
generalizing Theorem~\ref{thm:LRsemi}, which we explain now.

An equivariant
tableau is \emph{semistandard} if the box labels weakly increase along rows (left to right), and all labels strictly increase down
columns.  A single edge may be labeled by a \emph{set} of integers, without repeats; the smallest of them must be strictly greater than the
label of the box above, and the largest must be strictly less than the label of the box below.

\begin{example} Below is an equivariant semistandard Young tableau on $(4,2,2)/(2,1)$.
\[
\begin{picture}(100,60)
\put(20,40){$\tableauL{{\ }&{\ }&{1 }&{1 }\\{ \ }&{1}\\{2}&{3}}$}
\put(60,37){$2,3$} 
\put(83,37){$2$}
\end{picture}
\]
The content of this tableau is $(3,3,2)$.
\qed
\end{example}

Let ${\eqssyt}(\nu/\lambda)$ be the set of all equivariant semistandard Young tableaux of shape $\nu/\lambda$.
A tableau $T\in {\eqssyt}(\nu/\lambda)$ is \emph{ballot} if, for every column $c$ and every label $\ell$, 
\[
  (\#\text{ $\ell$'s weakly right of column }c ) \geq (\#\text{$(\ell+1)$'s weakly right of column }c ).
\]

Given a tableau $T\in \eqssyt(\nu/\lambda)$, a (box or edge) label $\ell$ is \emph{too high} if it appears weakly above the upper edge of 
a box in row $\ell$.  In the above example, all edge labels are too high. (When there are no edge labels, the semistandard and lattice 
conditions imply no box label is too high, but in general the three conditions are independent.)

Suppose an edge label $\ell$ lies on the bottom edge of a box $\xbox$ in row $i$.  Let $\rho_\ell(\xbox)$ be the number of times $\ell$ 
appears as a (box or edge) label strictly to the right of $\xbox$.  We define
\begin{equation}
\label{eqn:apfactor}
  \factor(\ell,\xbox) = t_{{\dist}(\xbox)} - t_{{\dist}(\xbox)+i-\ell+1+\rho_\ell(\xbox)}.
\end{equation}
When the edge label is not too high, this is always of the form $t_p-t_q$, for $p<q$.  (In particular, it is nonzero.)  Finally, define 
\begin{equation}
\label{eqn:apwt}
  {\sf Eballotwt}(T) = \prod \factor(\ell,\xbox),
\end{equation}
the product being over all edge labels $\ell$.

\begin{theorem}[Edge labeled ballot rule {\cite[Theorem~3.1]{Thomas.Yong}}] \label{thm:ty}
$C_{\lambda,\mu}^\nu = \sum_T {\sf Eballotwt}(T)$, 
where the sum is over all $T\in {\eqssyt}(\nu/\lambda)$
of content $\mu$ that are lattice and have no label which is too high.
\end{theorem}

\begin{example}\label{ex:eqLRcountLat} In Example \ref{ex:eqLRcountTab}, we saw 
 $C_{(2,2),(2,1)}^{(3,3)}=(t_2-t_5)+(t_1-t_4)$.
Below has $x\in\Lambda$ is filled with ${\dist}(\xbox)$. 
\[\tableauS{{2 }&{3 }&{ 4 }\\{1 }&{2 }&{ 3 }}\]

We see $T_1,T_2$ below are those $T\in {\eqssyt}((3,3)/(2,2))$
of content $(2,1)$ that are lattice and have no label which is too high.
\[
\begin{picture}(200,50)
\put(0,25){$T_1=\tableauL{{\ }&{\ }&{ 1 }\\{\ }&{ \ }&{2}}$}
\put(33,3){$1$}
\put(125,25){$T_2=\tableauL{{\ }&{\ }&{ 1 }\\{\ }&{ \ }&{2}}$}
\put(180,3){$1$}
\end{picture}
\]

Thus we see as Theorem \ref{thm:mainequiv} states, since for both cases $p_\ell(\xbox)=p_1(\xbox)=1$, 
\begin{align*}
    C_{(2,2),(2,1)}^{(3,3)}&={\sf Eballotwt}(T_1)+{\sf Eballotwt}(T_2)\\
    &=
\factor(1,(2,1))+\factor(1,(2,2))\\
&=(t_1-t_{1+2-1+1+1})+(t_2-t_{2+2-1+1+1})\\
&=(t_1-t_4)+(t_2-t_5),
\end{align*}
so this rule agrees.
\end{example}

\section{Nonvanishing of Littlewood-Richardson coefficients, Saturation and Horn inequalities}\label{sec:Horn}

\begin{center}
\emph{For which triples of partitions $(\lambda,\mu,\nu)$ does $c_{\lambda,\mu}^{\nu}\neq 0$?}
\end{center} 

The importance of this question
comes from a striking equivalence to a 19th century question about linear algebra. This equivalence was first suggested as a question by R.~C.~Thompson
in the 1970s; see R.~Bhatia's survey \cite[pg.~308]{Bhatia} 
(another survey on this topic is W.~Fulton's \cite{Fultona}). Suppose $A,B,C$ are three $r\times r$ Hermitian matrices and $\lambda,\mu,\nu\in {\mathbb R}^r$ are the respective
lists of eigenvalues (written in decreasing order). The \emph{eigenvalue problem for Hermitian matrices} asks
\begin{center}
\emph{Which eigenvalues $(\lambda,\mu,\nu)$ can occur if $A+B=C$?}
\end{center}
After work of H.~Weyl, K.~Fan, V.~B.~Lidskii-H.~Weilandt and others, A.~Horn recursively defined a list of inequalities on triples $(\lambda,\mu,\nu) \in {\mathbb R}^{3r}$.
He conjectured that these give a complete solution to the eigenvalue problem \cite{Horn}.  That these inequalities (or equivalent ones) are necessary
has been proved by several authors, including B.~Totaro \cite{Totaro} and A.~Klyachko \cite{Klyachko}. A.~Klyachko also established that his list of inequalities is sufficient, giving the first solution to the eigenvalue problem. 

Also, A.~Kylachko showed that his inequalities give an asymptotic solution to the problem of which Littlewood-Richardson coefficients $c_{\lambda,\mu}^{\nu}$ are nonzero.  That is, suppose $\lambda,\mu,\nu$ are partitions with at most $r$ parts.  A.~Klyachko proved that if $c_{\lambda,\mu}^{\nu}\neq 0$, then
$(\lambda,\mu,\nu)\in {\mathbb Z}_{\geq 0}^{3r}$ satisfies his inequalities. Conversely, he showed that if
$(\lambda,\mu,\nu)\in {\mathbb Z}_{\geq 0}^{3r}$ satisfy his inequalities then
$c_{N\lambda,N\mu}^{N\nu}\neq 0$ for some $N\in {\mathbb N}$.  (Here, $N\lambda$ is the partition with each part of $\lambda$ stretched by a factor of $N$.)
Subsequently, A.~Knutson-T.~Tao \cite{Knutson.Tao:99} sharpened the last statement, and established:
\begin{theorem}[Saturation theorem]\label{thm:usualsat} $c_{\lambda,\mu}^{\nu}\neq 0$ if and only if
$c_{N\lambda,N\mu}^{N\nu}\neq 0$ for any $N\in {\mathbb N}$.  
\end{theorem}
Combined with \cite{Klyachko}, it follows that A.~Klyachko's
solution agrees with A.~Horn's conjectured solution.\footnote{The proof in a preprint version of \cite{Knutson.Tao:99} used the polytopal
Littlewood-Richardson rule of Bernstein-Zelevinsky; we refer to the survey of A.~Buch \cite{Buch:after}. The published proof is formulated in terms of the \emph{Honeycomb model}. It is an easy exercise to prove
the ``$\Rightarrow$'' of the equivalence using either of the Littlewood-Richardson rules found in Section~\ref{sec:Young}. A proof of the converse 
using such rules would be surprising. Another solution, due to H.~Derksen-J.~Weyman \cite{DW} was given in the
setting of \emph{semi-invariants of quivers}.}
Let $[r]:=\{1,2,\ldots r\}$.  For any
\[I=\{i_1<i_2< \cdots<i_d\}\subseteq [r]\]
define the partition
\[
 \itop(I):=(i_d-d\geq \cdots\geq i_2-2 \geq i_1-1).
\]
This bijects subsets of $[r]$ of cardinality $d$ with partitions whose Young
diagrams are contained in a $d\times(r-d)$ rectangle. The following combines the main
results of \cite{Klyachko, Knutson.Tao:99}:

\begin{theorem}\label{thm:classicalHorn}(\cite{Klyachko}, \cite{Knutson.Tao:99})
Let $\lambda,\mu,\nu$ be partitions with at most $r$ parts such that
\begin{equation}
\label{eqn:classicalcond}
|\lambda|+|\mu|=|\nu|.
\end{equation}
The following are equivalent:
\begin{enumerate}
\item $c_{\lambda,\mu}^{\nu}\neq 0$.
\item For every $d<r$, and every triple of subsets $I,J,K\subseteq [r]$ of cardinality $d$ such that $c_{\itop(I),
\itop(J)}^{\itop(K)}\neq 0$, we have
\begin{equation}\label{eq:ineq}\sum_{i\in
I}\lambda_i+\sum_{j\in J}\mu_j\geq \sum_{k\in K}\nu_k.\end{equation}
\item There exist $r\times r$ Hermitian matrices $A,B,C$ with eigenvalues $\lambda,\mu,\nu$ such that $A+B=C.$
\end{enumerate}
\end{theorem}

\begin{remark}\label{Horn:converse}
The logic of the proof of Theorem~\ref{thm:classicalHorn} in \cite{Knutson.Tao:99} is to show Theorem~\ref{thm:usualsat}. 
In fact, Theorem~\ref{thm:usualsat} also follows from the equivalence (1)$\iff$(2) of 
Theorem~\ref{thm:classicalHorn}. This is since the Horn inequalities from (2) are homogeneous. This point seems to have been first noted
in  P.~Belkale's \cite{Belkale} which moreover gives a geometric proof of Theorem~\ref{thm:classicalHorn}. \qed
\end{remark}

\begin{remark}\label{remark:LRsemigroup}
The equivalence (1)$\iff$(2) immediately implies the semigroup property of ${\sf LR}_r$ (Corollary~\ref{cor:LRsemigroup}). \qed
\end{remark}

\begin{remark}
\label{remark:minimal}
P.~Belkale's doctoral thesis \cite{Belkale:thesis} (published in \cite{Belkale:thesispub}) shows that a much smaller list of inequalities than
those in Theorem~\ref{thm:classicalHorn}(2) suffice. Namely, replace the condition ``$c_{\tau(\lambda),\tau(\mu)}^{\tau(\nu)}\neq 0$''
with ``$c_{\tau(\lambda),\tau(\mu)}^{\tau(\nu)}\neq 1$''. A.~Knutson-T.~Tao-C.~Woodward \cite{KTW:JAMSII} showed that
the inequalities in this shorter list are minimal, \emph{i.e.}, none can be dispensed with.  \qed
\end{remark}

\begin{remark}
\label{remark:LRfg} Theorem~\ref{thm:classicalHorn} gives a different proof that ${\sf LR}_r$ is finitely generated as a semigroup.
We will give the argument in the proof of Proposition~\ref{prop:EqLRfinite}, below.
\end{remark}

While Theorem~\ref{thm:classicalHorn} characterizes nonvanishing of $c_{\lambda,\mu}^{\nu}$, the inequalities
are recursive and non-transparent to work with. The Littlewood-Richardson rules of Section~\ref{sec:Young}
require one to search for a valid tableau in a possibly large search space. 
K.~Purbhoo \cite{Purbhoo:root} (see also \cite{Purbhoo:rootgeneral}) developed a general and intriguing \emph{root game}, which
in the case of Grassmannians can be ``won'' if and only if $c_{\lambda,\mu}^{\nu}\neq 0$. 

Theorem~\ref{thm:usualsat} permits a determination of the
\emph{formal} computational
complexity of the nonvanishing decision problem ``$c_{\lambda,\mu}^{\nu}\neq 0$'' (in the bit length of the input $(\lambda,\mu,\nu)$, where 
one assumes arithmetic operations take constant time). This was resolved independently by T.~McAllister-J.~De Loera \cite{DM}
and K.~D.~Mulmuley-H.~Narayanan-M.~Sohoni, \cite{Mulmuley}, by a neat argument that combines Theorem~\ref{thm:usualsat} with celebrated
developments in linear programming:

\begin{theorem}
\label{thm:LRinP}
The decision problem of determining if $c_{\lambda,\mu}^{\nu}\neq 0$ is in the class ${\sf P}$ of polynomial problems.
\end{theorem}
\begin{proof}
By Theorem~\ref{thm:polytopalLR}, $c_{\lambda,\mu}^{\nu}\neq 0$ if and only if the polytope ${\mathcal P}_{\lambda,\mu}^{\nu}$ has a lattice point.
Clearly, if ${\mathcal P}_{\lambda,\mu}^{\nu}\neq \emptyset$, it has a rational vertex. In this case, a dilation $N{\mathcal P}_{\lambda,\mu}^{\nu}$ contains a lattice point. One checks from the definitions that
$N{\mathcal P}_{\lambda,\mu}^{\nu}={\mathcal P}_{N\lambda,N\mu}^{N\nu}$,
which means $c_{N\lambda,N\mu}^{N\nu}\neq 0$. Thus, by Theorem~\ref{thm:usualsat},
\[c_{\lambda,\mu}^{\nu}\neq 0\iff c_{N\lambda,N\mu}^{N\nu}\neq 0 \iff
{\mathcal P}_{\lambda,\mu}^{\nu}\neq \emptyset.\] 
To determine if ${\mathcal P}_{\lambda,\mu}^{\nu}\neq \emptyset$, one 
needs to decide feasiblity of any linear programming problem involving ${\mathcal P}_{\lambda,\mu}^{\nu}$. One appeals to  ellipsoid/interior point methods for polynomiality.\footnote{The Klee-Minty cube shows that the practically efficient simplex method has
exponential worst-case complexity.} Actually,
our inequalities are of the form
$A{\bf x}\leq {\bf b}$ where the vector ${\bf b}$ is integral and
the entries of $A$ are from $\{-1,0,1\}$.
Hence our polytope is \emph{combinatorial} and so one achieves a
strongly polynomial time complexity using \'E.~Tardos' algorithm; see
\cite{anothertardos, Tardos}. 
\end{proof}

In contrast, H.~Narayanan \cite{Narayanan} proved that counting $c_{\lambda,\mu}^{\nu}$ is a $\#${\sf P}-complete problem in L.~Valiant's complexity theory
of counting problems \cite{Valiant}. In particular, this means that no polynomial time algorithm for computing $c_{\lambda,\mu}^{\nu}$ can exist unless
${\sf P}={\sf NP}$ (it is widely expected that ${\sf P}\neq {\sf NP}$). It is curious that the counting problem is (presumably) hard, whereas the nonzeroness version is polynomial time. This already
occurs for the original $\#{\sf P}$-complete problem from \cite{Valiant}, \emph{i.e.}, to compute the permanent of an $n\times n$ matrix $M=(m_{ij})$ where
$m_{ij}\in \{0,1\}$. Now, determining if ${\rm per}(M)>0$ is equivalent to deciding the existence of matching in a bipartite graph that has incidence matrix $M$;
the algorithm of J.~Edmonds-R.~Karp provides the polynomial-time algorithm. 

\section{Equivariant nonvanishing, saturation, and Friedland's inequalities}\label{sec:Friedland}

We now turn to the equivariant analogues of results from the previous section.

\begin{center}
\emph{For which triples of partitions $(\lambda,\mu,\nu)$ does $C_{\lambda,\mu}^{\nu}\neq 0$?}
\end{center} 

In \cite{ARY} it was shown that this Schubert calculus question is also essentially equivalent to an eigenvalue problem. Recall that a Hermitian matrix $M$ \emph{majorizes} another Hermitian matrix $M'$ if $M-M'$ is positive semidefinite (its eigenvalues are all nonnegative).  In this case, we write $M\geq M'$.  S.~Friedland \cite{Friedland} studied the following question:
\begin{equation}\nonumber
\mbox{\emph{Which eigenvalues $(\lambda,\mu,\nu)$ can occur if $A+B \geq C$?}}
\end{equation}
His solution, given as linear inequalities, includes Klyachko's inequalities, a trace
inequality and some extra inequalities.  Later, W.~Fulton \cite{Fulton} proved the extra
inequalities are unnecessary, leading to a natural extension of the equivalence (2)$\iff$(3) of Theorem~\ref{thm:classicalHorn}.

One would like an extension of the equivalence with (1) of Theorem~\ref{thm:classicalHorn} as well. Now, D.~Anderson, E.~Richmond and the third author \cite{ARY} proved:

\begin{theorem}[Equivariant saturation]
\label{claim:mainARY}
$C_{\lambda,\mu}^{\nu}\neq 0$ if and only if $C_{N\lambda,N\mu}^{N\nu}\neq 0$ for any $N\in {\mathbb N}$.\footnote{In our notation, $C_{\lambda,\mu}^{\nu}$
depends on the indices $k$ and $n$ of the Grassmannian ${\sf Gr}_{k}({\mathbb C}^n)$. However, the edge labeled rule of Theorem~\ref{thm:mainequiv} has the property that
$C_{\lambda,\mu}^{\nu}$ is that for fixed $k$, the polynomial is independent of the choice of $n$ provided both 
are sufficiently large so that $\lambda,\mu,\nu\subseteq k\times (n-k)$. Indeed, the coefficients $C_{\lambda,\mu}^{\nu}$ for such values are the structure constants for
the Schubert basis in the graded inverse limit of equivariant cohomology rings under the standard embedding $\iota: {\sf Gr}_k({\mathbb C}^n)\hookrightarrow {\sf Gr}_{k}({\mathbb C}^{n+1})$; see \cite[Section~1.2]{ARY}.}
\end{theorem}

Actually, Theorem~\ref{claim:mainARY} is proved by establishing the the equivalence (1)$\iff$(2) below.

\begin{theorem}[\cite{ARY}, \cite{Friedland}, \cite{Fulton}]\label{thm:equivHorn}Let $\lambda,\mu,\nu$ be partitions with at most $r$ parts such that
\begin{equation}
\label{eqn:EqHorncond}
|\lambda|+|\mu|\geq |\nu| \mbox{\ and  \ }\max\{\lambda_i,\mu_i\}\leq \nu_i \mbox{\ for all $i\leq r$.}
\end{equation}
The following are equivalent:
\begin{enumerate}
\item $C_{\lambda,\mu}^{\nu}\neq 0$.
\item For every $d<r$, and every triple of subsets $I,J,K\subseteq [r]$ of cardinality $d$ such that $c_{\itop(I),
\itop(J)}^{\itop(K)}\neq 0$, we have $$\sum_{i\in I}\lambda_i+\sum_{j\in J}\mu_j\geq \sum_{k\in
K}\nu_k.$$

\item There exist $r\times r$ Hermitian matrices $A,B,C$ with eigenvalues $\lambda,\mu,\nu$ such that $A+B\geq C$.
\end{enumerate}
\end{theorem}

Theorem \ref{thm:equivHorn} states that the main inequalities controlling nonvanishing of $C_{\lambda,\mu}^{\nu}$ are just Horn's inequalities  (\ref{eq:ineq}).  

\begin{Remark}\label{remark:secondcond} The second condition in (\ref{eqn:EqHorncond}) is
unnecessary in Theorem~\ref{thm:classicalHorn} since it is already implied by
(\ref{eqn:classicalcond}) combined with \eqref{eq:ineq}. The condition is not required for the equivalence (2)$\iff$(3). However it is needed
for the equivalence with (1). For example, the $1\times 1$ matrices $A=[1], B=[1], C=[0]$ satisfy $A+B\geq C$, but $C_{(1),(1)}^{(0)}=0$.
This is not in contradiction with Theorem~\ref{thm:equivHorn} since $\max\{1,1\}\leq 0$ is violated.\qed
\end{Remark}

Just as in Remark~\ref{Horn:converse}, the equivalence (1)$\iff$(2) and the homogeneity of the Friedland-Fulton inequalities
implies Theorem~\ref{claim:mainARY}. On the other hand the equivalence (1)$\iff$(2) relies on the classical Horn theorem (Theorem~\ref{thm:classicalHorn})
and the edge labeled ballot tableau rule Theorem~\ref{thm:ty}. To give the reader a sense of the proof we now provide:

\noindent
\emph{Sketch of proof that (1)$\implies$(2):} Using Theorem~\ref{thm:ty}, it is an exercise to show that 
\begin{claim}\label{claim:goingdown} If $C_{\lambda,\mu}^{\nu}\neq 0$ and $|\nu|<|\lambda|+|\mu|$, then for any $s$ such that $|\nu|-|\lambda|\leq s<|\mu|$, there is a
$\mu^{\downarrow}\subset \mu$ with $|\mu^{\downarrow}|=s$ and $C_{\lambda,\mu^{\downarrow}}^{\nu}\neq 0$.
\end{claim}

Since $C_{\lambda,\mu}^{\nu}\neq 0$, by the claim (and induction) there exists $\lambda^{\downarrow}\subseteq \lambda$ such that $|\lambda^{\downarrow}|+|\mu|=|\nu|$ and $C_{\lambda^{\downarrow},\mu}^{\nu}\neq 0$. This latter number is a classical Littlewood-Richardson coefficient, we can apply Theorem~\ref{thm:classicalHorn} to conclude that for any triple $(I,J,K)$ with $c_{\tau(I),\tau(J)}^{\tau(K)}\neq 0$ one has
\[\sum_{i\in I}\lambda_i^{\downarrow}+\sum_{i\in J}\mu_j\geq \sum_{k\in K}\nu_k.\]
Now we are done since $\sum_{i\in I}\lambda_i\geq \sum_{i\in I}\lambda_i^{\downarrow}$.\qed

The converse (2)$\implies$(1) uses another exercise that can be proved using Theorem~\ref{thm:ty}:

\begin{claim}\label{claim:goingup}
If $C_{\lambda,\mu}^{\nu}\neq 0$ then $C_{\lambda,\mu^{\uparrow}}^{\nu}\neq 0$ for any $\mu\subset \mu^{\uparrow}\subseteq \nu$
\end{claim}
See \cite[Section~2]{ARY} for proofs of Claim~\ref{claim:goingdown} and Claim~\ref{claim:goingup}.

\begin{remark}
Just as with Theorem~\ref{thm:classicalHorn}, the list of inequalities in Theorem~\ref{thm:equivHorn}(2) contain redundancies.
Building from the results discussed in Remark~\ref{remark:minimal}, W.~Fulton~\cite{Fulton} shows that one can also replace the 
``$c_{\tau(\lambda),\tau(\mu)}^{\tau(\nu)}\neq 0$'' with ``$c_{\tau(\lambda),\tau(\mu)}^{\tau(\nu)}=1$''. W.~Fulton's work  shows
that if the second condition in (\ref{eqn:EqHorncond}) is ignored, the inequalities are minimal for the equivalence
(2)$\iff$(3). 
\qed
\end{remark}

Let ${\sf EqLR}_r=\{(\lambda,\mu,\nu):\ell(\lambda),\ell(\mu),\ell(\nu)\leq r, C_{\lambda,\mu}^{\nu}\neq 0\}$.

\begin{corollary}
${\sf EqLR}_r$ is a semigroup.
\end{corollary}
\begin{proof}
Just as Corollary~\ref{cor:LRsemigroup} clearly follows from the first two equivalences of Theorem~\ref{thm:classicalHorn} (Remark~\ref{remark:LRsemigroup}), the present claim holds by the first two equivalences of Theorem~\ref{thm:equivHorn}. However,
we can also prove this directly from Claims~\ref{claim:goingdown} and~\ref{claim:goingup}: By applying Claim~\ref{claim:goingdown} there exists $\lambda^\circ\subseteq \lambda$ and $\alpha^{\circ}\subseteq \alpha$ such that
$C_{\lambda^{\circ},\mu}^{\nu}=c_{\lambda^{\circ},\mu}^{\nu}\neq 0$ and
$C_{\alpha^\circ,\beta}^{\gamma}=c_{\alpha^{\circ},\beta}^{\gamma}\neq 0$. By Corollary~\ref{cor:LRsemigroup}, 
$(\lambda^\circ+\alpha^{\circ},\mu+\beta,\nu+\gamma)\in {\sf LR}_r$. Now, since $\lambda^{\circ}+\alpha^{\circ}\subseteq \lambda+\alpha\subseteq
\nu+\gamma$, we can apply Claim~\ref{claim:goingup} to conclude $(\lambda+\alpha,\mu+\beta,\nu+\gamma)\in {\sf EqLR}_r$, as desired.
\end{proof}

\begin{proposition}
\label{prop:EqLRfinite}
${\sf EqLR}_r$ is finitely generated.
\end{proposition}

The argument we give is based on discussion with S.~Fomin and A.~Knutson. It applies \emph{mutatis mutandis} to
prove that ${\sf LR}_r$ is finitely generated:

\begin{proof}
Since the inequalities from Theorem~\ref{thm:equivHorn} are finite in number, and each inequality has its bounding hyperplane
containing the origin, the set ${\mathcal C}\subseteq {\mathbb R}^{3r}$ they define is a polyhedral cone. Since the inequalities
have rational coefficients, by definition, ${\mathcal C}$ is rational.  Moreover, ${\mathcal C}$ is also clearly pointed, \emph{i.e.}, 
${\mathcal C}\cap -{\mathcal C}=\{{\bf 0}\}$. Now apply \cite[Theorem~16.4]{Schrijver}.
\end{proof}
Naturally, one would like a solution for the generalization of Problem~\ref{problem:zelev} to ${\sf EqLR}_r$.

Theorem~\ref{claim:mainARY} together with Theorem~\ref{thm:ty}, A.~Adve together with the first and third authors \cite{ARY} prove a generalization of Theorem~\ref{thm:LRinP}:

\begin{theorem}[\cite{ARY}]
\label{thm:equivLRinP}
The decision problem of determining $C_{\lambda,\mu}^{\nu}\neq 0$ is in the class ${\sf P}$ of polynomial problems.
\end{theorem}
\noindent
\emph{Sketch of proof:} Using Theorem~\ref{thm:LRinP} one can construct a polytope ${\mathcal Q}_{\lambda,\mu}^{\nu}$ analogous to ${\mathcal P}_{\lambda,\mu}^{\nu}$. The main property is that $C_{\lambda,\mu}^{\nu}\neq 0$ if and only if ${\mathcal Q}_{\lambda,\mu}^{\nu}$ contains a lattice point (in particular,
in contrast to Theorem~\ref{thm:polytopalLR}, the \emph{number} of lattice points of ${\mathcal Q}_{\lambda,\mu}^{\nu}$ is not well-understood). The remainder
of the proof proceeds exactly as in the proof of Theorem~\ref{thm:LRinP}, except that we use Theorem~\ref{claim:mainARY} in place of 
Theorem~\ref{thm:usualsat}.\qed

For the remainder of this section, assume $C_{\lambda,\mu}^{\nu}$ is expressed (uniquely) as a polynomial in the variables $\beta_i:=t_{i}-t_{i+1}$. We now state
a few open problems/conjectures introduced in \cite{RYY}, for the special case of Grassmannians.

A refinement of the nonvanishing question is:
\begin{question}\label{question:refinednonvanishing}
What is the computational complexity of determining if $[\beta_1^{i_1}\cdots \beta_{n-1}^{i_{n-1}}]C_{\lambda,\mu}^{\nu}\neq 0$?
\end{question} 

This question concerns the Newton polytope of $C_{\lambda,\mu}^{\nu}$.
Recall, the \emph{Newton polytope} of 
\[f=\sum_{(n_1,\ldots,n_r)\in {\mathbb Z}_{\geq 0}^r} c_{n_1,\ldots,n_r}\prod_{j=1}^r \alpha_j^{n_j}\in {\mathbb R}[\alpha_1,\ldots,\alpha_r]\] 
is ${\sf Newton}(f):={\sf conv}\{(n_1,\ldots,n_r): c_{n_1,\ldots,n_r}\neq 0\}\subseteq {\mathbb R}^r$. 
$f$ has \emph{saturated Newton polytope} (SNP) \cite{MTY} if
$c_{n_1,\ldots,n_r}\neq 0 \iff (n_1,\ldots,n_r)\in {\sf Newton}(f)$. 
\begin{conjecture}
\label{conj:SNP}
$C_{\lambda,\mu}^{\nu}$ has SNP.
\end{conjecture}

This raises the question:

\begin{problem}
\label{prob:halfspace}
Give a half space description of ${\sf Newton}(C_{\lambda,\mu}^{\nu})$.
\end{problem}

A proof of Conjecture~\ref{conj:SNP} together with any reasonable solution to
Problem~\ref{prob:halfspace} would imply that the decision problem in Question~\ref{question:refinednonvanishing}
is in the computational complexity class ${\sf NP}\cap {\sf coNP}$. This would strongly suggest that the decision problem is not
{\sf NP}-{\sf complete}, and in fact suggest the problem is in {\sf P}. We refer the reader to \cite[Section~1]{Anshul.Robichaux.Yong}
for elaboration on these points.

\section{Maximal orthogonal and Lagrangian Grassmannians}\label{sec:PQstuff}

\subsection{Goals in the sequel} Beyond Grassmannians ${\sf Gr}_k({\mathbb C}^n)$, 
the \emph{maximal orthogonal Grassmannians} and \emph{Lagrangian Grassmannians} have been of significant interest. Their classical (non-equivariant) 
Schubert calculus shares many analogies with the Grassmannian case. They concern the $Q$-Schur polynomials
of I.~Schur \cite{Schur}, and the tableau combinatorics of D.~Worley \cite{Worley}, B.~Sagan \cite{Sagan} and J.~Stembridge \cite{Stembridge}.
Although these combinatorial results were originally developed to study projective representations of symmetric group, the connection to Schubert calculus of these
spaces was established by P.~Pragacz \cite{Pragacz}. 

It is therefore natural to seek extensions of the results from Sections~1-6. Discussion of efforts toward this goal occupy the remainder of this work.

\subsection{Definition of the spaces}
Consider the two classical Lie groups of non-simply laced type: ${\sf G}={\sf SO}_{2n+1}({\mathbb C})$ and ${\sf G}={\sf Sp}_{2n}({\mathbb C})$. These are the automorphism groups of a non-degenerate bilinear form $\langle \cdot,\cdot\rangle$. In the former case, $\langle\cdot,\cdot\rangle$ is symmetric, and on $W={\mathbb C}^{2n+1}$.
In the latter case,  $\langle\cdot,\cdot\rangle$ is skew-symmetric, and on $W={\mathbb C}^{2n}$. A subspace $V\subseteq W$ is called \emph{isotropic} if $\langle v_1,v_2\rangle=0$ for all $v_1,v_2\in V$. The maximum dimension of an isotropic subspace of $W$ is $n$. 

Let $Y={\sf OG}(n,2n+1)$ be the \emph{maximal orthogonal Grassmannian} of $n$-dimensional isotropic subspaces of ${\mathbb C}^{2n+1}$; this space has an action of ${\sf G}={\sf SO}_{2n+1}({\mathbb C})$. Similarly, let $Z={\sf LG}(n,2n)$
be the \emph{Lagrangian Grassmannian} of $n$-dimensional isotropic subspaces of ${\mathbb C}^{2n}$. In either case,
the (opposite) Borel subgroup ${\sf B}_{-}\leq {\sf G}$ consists of the lower triangular matrices in ${\sf G}$. 
The maximal torus ${\sf T}$ are the diagonal matrices in ${\sf G}$.
Just as in the case of the Grassmannian, the corresponding ${\sf B}_{-}$ acts on $Y$ (resp.~$Z$) with finitely
many orbits $Y_{\lambda}^{\circ}$ (resp.~$Z_{\lambda}^{\circ}$); these are the \emph{Schubert cells}. 

In both cases, the Schubert cells and thus Schubert varieties $Y_{\lambda}=\overline{Y_{\lambda}^{\circ}}$ (resp.~$Z_{\lambda}=\overline{Z_{\lambda}^{\circ}}$) are indexed by strict partitions fitting inside the shifted staircase 
\[\rho_n=(n,n-1,n-2,\ldots,3,2,1).\]
A \emph{strict partition} is an integer partition $\lambda=(\lambda_1>\lambda_2>\ldots > \lambda_{\ell}>0)$. Identify $\lambda$ with its shifted shape,
which is the usual Young diagram (in English notation) but where the $i$-th row from the top is indented $i-1$ many spaces. We refer to, \emph{e.g.}, \cite[Section~6]{Ikeda.Naruse} and the references therein for additional details.

Let $\sigma_{\lambda}(Y)\in H^{2|\lambda|}(Y)$ be the Poincar\'e dual to $Y_{\lambda}$. These \emph{Schubert
classes} form a ${\mathbb Z}$-linear basis of $H^{*}(Y)$, and we define the structure constants by
\[\sigma_{\lambda}(Y)\smallsmile \sigma_{\mu}(Y)=\sum_{\nu\subseteq \rho_n} o_{\lambda,\mu}^{\nu}\sigma_{\nu}(Y).\]
Similarly, the Schubert classes $\sigma_{\lambda}(Z)$ form a ${\mathbb Z}$-linear basis of $H^{*}(Z)$, and
\[\sigma_{\lambda}(Z)\smallsmile \sigma_{\mu}(Z)=\sum_{\nu\subseteq \rho_n} l_{\lambda,\mu}^{\nu} \sigma_{\nu}(Z).\]

\subsection{Schur $P-$ and $Q-$ functions;  P.~Pragacz's theorem}

Let
\[q_r(x_1,\ldots,x_n)=2\sum_{i=1}^n x_i^r\prod_{i\neq j}\frac{x_i+x_j}{x_i-x_j}.\]
This is clearly symmetric and in fact a polynomial. Next, set
\[Q_{(r,s)}=q_r q_s+2\sum_{i=1}^s (-1)^i q_{r+i}q_{s-i}.\]
Recall that the \emph{Pfaffian} of a $2t\times 2t$ skew-symmetric matrix $M=(m_{ij})$ is
\[{\sf Pf}(M)=\sum_{\sigma\in {\mathfrak S}_{2t}} {\sf sgn}(\sigma)\prod_{i=1}^{t} m_{\sigma(2i-1), \sigma(2i)},\]
where $\sigma$ satisfies $\sigma(2i-1)<\sigma(2i)$ for $1\leq i \leq m$ and $\sigma(1)<\sigma(3)<\ldots<\sigma(2i-3)<\sigma(2t-1)$.
Then for $\lambda=(\lambda_1>\lambda_2>\ldots>\lambda_{\ell}>0)$, the \emph{Schur $Q-$ function} \cite{Schur} is
\begin{equation}
\label{eqn:Qpfaff}
Q_{\lambda}={\sf Pf}(Q_{(\lambda_i,\lambda_j)}).
\end{equation}
If $\ell(\lambda)$ is odd, we add a $0$ at the end. The Schur $Q-$ functions linearly span the subalgebra $\Gamma\subset {\sf Sym}$ generated by the $q_i$'s. In fact, $Q_{(i)}=q_i$.
The \emph{Schur $P-$ function} is
\begin{equation}
\label{eqn:PQ}
P_{\lambda}:=2^{-\ell(\lambda)}Q_{\lambda}.
\end{equation}

P.~Pragacz \cite{Pragacz} proved that the Schur $P-$ and Schur $Q-$ polynomials represent the Schubert classes of $Y$ and $Z$ respectively.
That is, 
\[H^{*}(Y)\cong \Gamma/J,\]
where $J$ is the ideal $\langle P_{\lambda}:\lambda\not\subseteq \rho_n\rangle$, and 
$\sigma_{\lambda}(Y)$ maps to $P_{\lambda}+J$ under this isomorphism. Moreover,
\[P_{\lambda} P_{\mu}=\sum_{\nu} o_{\lambda,\mu}^{\nu}P_{\nu}.\] 
Similarly,
\[Q_{\lambda} Q_{\mu}=\sum_{\nu} l_{\lambda,\mu}^{\nu} Q_{\nu}.\] 

J.~R.~Stembridge \cite{Stembridge} proved 
\[P_{\lambda}=\sum_T x^T,\]
where the sum is over semistandard fillings of the shifted shape $\lambda$. That is, fill each box of $\lambda$ with a label
from the ordered set $1'<1<2'<2<3'<3<\ldots$ such that the rows and columns are weakly increasing, two $i'$-s cannot appear in the same row and two $i$-s cannot appear in the same column. Moreover, there are no primed entries on the main diagonal. For example, if $\lambda=(2,1)$, the shifted semistandard tableaux are
\[\tableauS{1&1\\& 2} \ \ \ \ \tableauS{1&2'\\&2} \ \ \ \  \tableauS{1&1\\&3} \ \ \ \   \tableauS{1&2'\\&3} \ \ \ \  \tableauS{1&2\\&3}\ \ \ \  \tableauS{1&3'\\&3}\ \ \ \  \tableauS{2&2\\&3}\ \ \ \  \tableauS{2&3'\\&3} \ \ \ \ \cdots \] 
Hence $P_{(2,1)}(x_1,x_2,x_3)=x_1 x_2^2 +x_1 x_2^2 +x_1^2 x_3 +x_1 x_2 x_3 +x_1 x_2 x_3 +x_1 x_3^2 +x_2^2 x_3 +x_2 x_3^2 + \cdots $. 
If we allow primed entries on the diagonal, we get the Schur $Q-$ function. It is easy to see that this definition satisfies (\ref{eqn:PQ}).

\subsection{Shifted Littlewood-Richardson rules}
D.~Worley \cite{Worley} introduced a \emph{jeu de taquin} theory for shifted shapes. A standard tableau $T$ of shifted skew shape $\nu/\lambda$ is a filling
of $\lambda$ with the labels $1,2,3,\ldots,|\nu/\lambda|$ that is increasing along rows and columns. Let ${\sf shSYT}(\nu/\lambda)$ denote the set of these tableaux. The notions of slides and rectification
are just as in the unshifted case, using (J1) and (J2). With this, the exact analogues of Theorems~\ref{thm:firstjdt} and~\ref{thm:secondjdt}
hold and one can define ${\sf ShEjdt}$ and ${\sf shRect}$ etc, in the obvious manner. Indeed, one has the following combinatorial rule for $o_{\lambda,\mu}^{\nu}$:

\begin{theorem}[Shifted jeu de taquin Littlewood-Richardson rule]
\label{thm:shiftedjdt}
Fix $U\in {\sf shSYT}(\mu)$. Then
\[o_{\lambda,\mu}^{\nu}=\#\{T\in  {\sf shSYT}(\nu/\lambda):{\sf shRect}(T)=U\}.\]
\end{theorem}

\begin{example}\label{ex:LRshiftedTab}
Let $\lambda=(3,1),\mu=(3,1),\nu=(4,3,1)$. The following are the $2$ shifted tableaux of shape $\nu/\lambda$ that rectify to $U = \tableauS{{1}&{2}&{3}\\ &{4}} $. Thus, $o_{\lambda,\mu}^{\nu}=2$.
\begin{gather*}
\tableauS{{\ }&{\  }&{\ }&{1}\\  & {\ }&{2}&{3}&\\  &  &{4}&}
\qquad
\tableauS{{\ }&{\  }&{\ }&{3}\\  & {\ }&{1}&{4}&\\  &  &{2}&}
\end{gather*}
\end{example}

 Define the {\it reading word} of a (possibly skew) shifted tableaux $T$ to be the word obtained by reading the rows of $T$ from left to right starting with the bottom row. For a word $w=w_1 w_2 \ldots w_n$, define
 ${\sf JS}_i(j)$ for $1\leq j\leq 2n, i\geq 1$, depending on $w$: 
\[{\sf JS}_i(j):=\text{number of occurrences of $i$ among $w_{n-j+1},\ldots,w_n$}, \ 0\leq j\leq n,\]
and
\[{\sf JS}_i(n+j):={\sf JS}_i(n)+\text{number of occurrences of $i'$ among $w_1,\ldots,w_j$}, \ 0<j\leq n.\]

The word $w$ is {\it proto-ballot} if, when ${\sf JS}_i(j)={\sf JS}_{i-1}(j)$, both of these statements hold:
\begin{center}
$w_{n-j}\neq i,i'$,   \  if $0\leq j<n;$\\ 
$w_{j-n+1}\neq i-1,i'$, \  if $n\leq j <2n$.
\end{center}

Let $|w|$ be the word obtained from $w$ by removing all primes. Now, $w$ is {\it ballot} if it is {\it proto-ballot} and the leftmost $i$ of $|w|$ is unprimed in $w$ for all $i$. In his work on projective representation theory of symmetric groups, J.~Stembridge \cite[Theorem 8.3]{Stembridge} gave the following semistandard analogue of Theorem~\ref{thm:shiftedjdt}.
\begin{theorem}[Shifted ballot Littlewood-Richardson rule]\label{thm:jrsrule}
$o_{\lambda,\mu}^{\nu}$ counts the number of shifted semistandard tableaux of shape $\nu/\lambda$ of content $\mu$
that are ballot.
\end{theorem}

\begin{example}\label{ex:ShLRballotTab}
Let $\lambda=(3,1),\mu=(3,1),\nu=(4,3,1)$, then the following are the only $2$ shifted semistandard tableaux of shape $\nu/\lambda$ of content $\mu$
that are ballot.
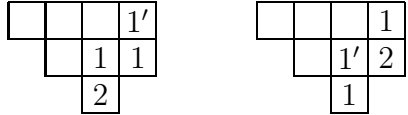
\begin{gather*}
\tableauS{{\ }&{\  }&{\ }&{1'}\\  & {\ }&{1}&{1}&\\  &  &{2}&}
\qquad
\tableauS{{\ }&{\  }&{\ }&{1}\\  & {\ }&{1'}&{2}&\\  &  &{1}&}
\end{gather*}
\end{example}

It follows from (\ref{eqn:PQ}) that
\begin{equation}
\label{eqn:PQAug11}
l_{\lambda,\mu}^{\nu}=2^{\ell(\nu)-\ell(\lambda)-\ell(\mu)}o_{\lambda,\mu}^{\nu}.
\end{equation}
Thus the above rules give a rule to compute $l_{\lambda,\mu}^{\nu}$ in a manifestly positive manner, as well.

\subsection{Nonvanishing} K.~Purbhoo-F.~Sottile \cite{PS:FPSAC, Purbhoo.Sottile} gave an extension of the Horn recursion to describe when $o_{\lambda,\mu}^{\nu}$ (or equivalently $l_{\lambda,\mu}^{\nu}$) is nonzero. Fix $n$ and $r$. Suppose $\alpha$ is a (ordinary) partition whose (unshifted) Young diagram is contained in $r\times (n-r)$. Let
\[I_{n}(\alpha):=\{n-r+1-\alpha_1,n-r+2-\alpha_2,\ldots,n-\alpha_r\}.\] 
Index the corners of $\rho_n$ top to bottom from $1$ to $n$. For $\lambda\subseteq \rho_n$, $0<r<n$ and $\alpha\subseteq r\times (n-r)$ let $[\lambda]_{\alpha}$ 
be the number of boxes of $\lambda$ that survive after crossing the rows to the right and columns above the corners indexed by
$I_{n}(\alpha)$. Define $\lambda^c$ to be the complement of $\lambda$
in $\rho_n$ (after reflecting), whereas $\alpha^c=\alpha^{\vee}$ is the rotation of the complement of $\alpha$ in $r\times (n-r)$.

\begin{example}\label{ex:PurbhooSottile}
Suppose that $n=6,r=3$ and $\alpha=(3,2,1)$. Then $$I_n(\alpha)=\{6-3+1-3,6-3+2-2,6-3+3-1\} = \{1,3,5\}.$$
Suppose $\lambda=(6,4,3,1)\subseteq \rho_n$. In the figure below, yellow boxes are the ones that are crossed out. Thus, $[\lambda]_{\alpha}=3$.
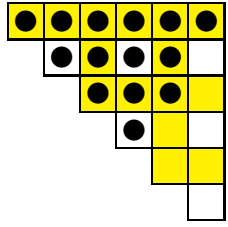
\begin{gather*}
\ytableausetup
{boxsize=1.1em}\begin{ytableau}
 *(yellow) \newmoon & *(yellow)\newmoon & *(yellow)\newmoon & *(yellow)\newmoon & *(yellow)\newmoon & *(yellow)\newmoon \\ 
 \none     & \newmoon  & *(yellow)\newmoon &\newmoon   & *(yellow)\newmoon  & \\
\none  & \none & *(yellow)\newmoon & *(yellow)\newmoon & *(yellow) \newmoon& *(yellow) \\
 \none & \none & \none &\newmoon & *(yellow) & \\
 \none & \none & \none & \none & *(yellow) & *(yellow) \\
 \none & \none & \none & \none & \none & 
\end{ytableau}
\end{gather*}
\end{example}

\begin{theorem}[K.~Purbhoo-F.~Sottile's theorem]\label{thm:PurbhooSottile}
For $\lambda, \mu,\nu\subseteq \rho_n$,
$o_{\lambda,\mu}^{\nu^c}\neq 0$ if and only if
\begin{itemize}
\item $|\lambda|+|\mu|+|\nu|=\dim Y=\binom{n+1}{2}$, and
\item for all $0<r<n$ and all $\alpha,\beta,\gamma\subset r\times (n-r)$ such that $c_{\alpha,\beta}^{\gamma^c}\neq 0$, one has
$[\lambda]_{\alpha}+[\mu]_{\beta}+[\nu]_{\gamma}\leq \binom{n+1-r}{2}$.
\end{itemize}
\end{theorem}

\begin{Remark}\label{remark:nosaturation}
The obvious analogue of saturation does \emph{not} hold. For example, take $\lambda=(2,1),\mu=(2),\nu=(3,2)$. Then $o_{\lambda,\mu}^{\nu} \neq 0$, but $o_{2\lambda,2\mu}^{2\nu}=o_{(4,2),(4)}^{(6,4)}=0$. \qed
\end{Remark}

Inspired by the complexity results concerning $c_{\lambda,\mu}^{\nu}$, we take this opportunity to pose:

\begin{problem}
Is the decision problem of determining if $o_{\lambda,\mu}^{\nu}\neq 0$ in the class ${\sf P}$ of polynomial time problems?
\end{problem}

Remark~\ref{remark:nosaturation} implies that the argument used in the proofs of Theorems~\ref{thm:LRinP} and~\ref{thm:equivLRinP}
cannot work.

\begin{problem}
Is counting $o_{\lambda,\mu}^{\nu}$ is in the class of $\#${\sf P}-complete problems?
\end{problem}

\section{Equivariant Schubert calculus of $Y$ and $Z$}\label{sec:YZ}

One is interested in the equivariant cohomology of $Y$ and $Z$. 
As with Grassmannians (Section~\ref{sec:Equiv}) , one has structure constants with respect to the Schubert basis,
\[ \xi_{\lambda}(Y)\cdot \xi_{\mu}(Y)=\sum_{\nu\subseteq \rho_n} O_{\lambda,\mu}^{\nu} \xi_{\nu}(Y)
 \text{ \ and \ $ \xi_{\lambda}(Z) \cdot \xi_{\mu}(Z) = \sum_{\nu\subseteq \rho_n} L_{\lambda,\mu}^{\nu} \xi_{\nu}(Z)$.}\]
If $|\lambda|+|\mu|=|\nu|$ then $O_{\lambda,\mu}^{\nu}=o_{\lambda,\mu}^{\nu}$ and $L_{\lambda,\mu}^{\nu}=l_{\lambda,\mu}^{\nu}$. For sake of brevity,
we refer to \cite{Ikeda.Naruse} and the references therein.

Suppose that $H_{\sf T}^{*}(pt)=\mathbb{Z}[t_1,\ldots,t_n]$. The general form of 
Theorem~\ref{thm:Graham} (see the attached footnote to that result) states that if $\gamma_1=t_1$ and for $i>1$, $\gamma_i=t_i-t_{i-1}$, then 
\[O_{\lambda,\mu}^{\nu}\in {\mathbb Z}_{\geq 0}[\gamma_1,\gamma_2,\ldots,\gamma_{n}].\]
Similarly, if $\alpha_1=2{t}_1,\alpha_2={t}_2-{t}_1 \ldots,\alpha_{n}={t}_n-{t}_{n-1}$ then
\[L_{\lambda,\mu}^{\nu}\in {\mathbb Z}_{\geq 0}[\alpha_1,\alpha_2,\ldots,\alpha_{n}].\]

\begin{problem}\label{problem:OG}
Give a combinatorial rule for $O_{\lambda,\mu}^{\nu}$ and/or $L_{\lambda,\mu}^{\nu}$.
\end{problem}

Naturally, we desire a rule in terms of shifted edge labeled tableaux. Such a rule (or any combinatorial rule) has eluded us.
The reader wishing to give Problem~\ref{problem:OG} a try might find Table~\ref{tab:prod} useful.

\begin{table}[t]\label{tabel:eqTypeBC}
\begin{tabular}{ |c|c|c|c|c|c| }
\hline
$\lambda$ & $\mu$ & $\nu$ & $O_{\lambda,\mu}^{\nu}$ & $ L_{\lambda,\mu}^{\nu}$  \\  
\hline\hline
$[1]$ & $[1]$ & $[1]$ & $\gamma_{1}$ & $\alpha_{1}$ \\
$[1]$ & $[1]$ & $[1]$ & $\gamma_{1}$ & $\alpha_{1}$ \\
$[1]$ & $[1]$ & $[2]$ & $1$ & $2$\\
$[1]$ & $[2]$ & $[2]$ & $\gamma_{1} + \gamma_{2}$ & $\alpha_{1} + 2 \alpha_{2}$\\
$[2]$ & $[1]$ & $[2]$ & $\gamma_{1} + \gamma_{2}$ & $\alpha_{1} + 2 \alpha_{2}$ \\
$[2]$ & $[2]$ & $[2]$ & $\gamma_{1} \gamma_{2} + \alpha_{2}^2$ & $\alpha_{1} \alpha_{2} + 2 \alpha_{2}^2$ \\
$[1]$ & $[2]$ & $[2, 1]$ & $1$ & $1$ \\
$[1]$ & $[2, 1]$ & $[2, 1]$ & $2\gamma_{1} + \gamma_{2}$ & $2\alpha_{1} + 2\alpha_{2}$ \\
$[2]$ & $[1]$ & $[2, 1]$ & $1$ & $1$ \\
$[2]$ & $[2]$ & $[2, 1]$ & $2\gamma_{1} + 2\gamma_{2}$ & $\alpha_{1} + 2\alpha_{2}$ \\
$[2]$ & $[2, 1]$ & $[2, 1]$ & $2\gamma_{1}^2 + 3\gamma_{1}\gamma_{2} + \gamma_{2}^2$ & $\alpha_{1}^2 + 3\alpha_{1}\alpha_{2} + 2\alpha_{2}^2$ \\
$[2, 1]$ & $[1]$ & $[2, 1]$ & $2\gamma_{1} + \gamma_{2}$ & $2\alpha_{1} + 2\alpha_{2}$ \\
$[2, 1]$ & $[2]$ & $[2, 1]$ & $2\gamma_{1}^2 + 3\gamma_{1}\gamma_{2} + \gamma_{2}^2$ & $\alpha_{1}^2 + 3\alpha_{1}\alpha_{2} + 2\alpha_{2}^2$ \\
$[2, 1]$ & $[2, 1]$ & $[2, 1]$ & $2\gamma_{1}^3 + 3\gamma_{1}^2\gamma_{2} + \gamma_{1}\gamma_{2}^2$ & $\alpha_{1}^3 + 3\alpha_{1}^2\alpha_{2} + 2\alpha_{1}\alpha_{2}^2$ \\

\hline
\end{tabular}

\caption{Table of products for $n=2$ \label{tab:prod}}
\end{table}

 Let $\tilde{L}_{\lambda,\mu}^{\nu}$ be the polynomial obtained from $L_{\lambda,\mu}^{\nu}$ after substitions $\alpha_1 \mapsto 2\gamma_1$ and $\alpha_i \mapsto \gamma_i$ for $i>1$. 
The following is a refinement of (\ref{eqn:PQAug11}):

\begin{theorem}[\emph{cf.} Theorem~1.1 of \cite{RYY}]\label{thm:typeBCpower2}
$O_{\lambda,\mu}^{\nu} = 2^{\ell(\nu)-\ell(\lambda)-\ell(\mu)} \tilde{L}_{\lambda,\mu}^{\nu}$
\end{theorem}

Thus, the $O_{\lambda,\mu}^{\nu}$ and $L_{\lambda,\mu}^{\nu}$ versions of Problem~\ref{problem:OG}
are essentially equivalent.

Theorem~\ref{thm:typeBCpower2} was stated in a weaker form as a conjecture in 
C.~Monical's doctoral thesis \cite[Conjecture~5.1]{Monical}. A proof of a generalization was given in \cite[Theorem~1.1]{RYY}. Below, we offer another proof
that uses a variation of the associativity recurrence alluded to at the end of Section~\ref{subsection:GKM}. This recurrence should be useful
to prove any guessed rule for $O_{\lambda,\mu}^{\nu}$ or $L_{\lambda,\mu}^{\nu}$, so we wish to explicate it here.

\begin{proof}
This will serve as the base case of the associativity recurrence below:
\begin{lemma}
$O_{\lambda,\mu}^{\lambda}=2^{\ell(\nu)-\ell(\lambda)-\ell(\mu)} \tilde{L}_{\lambda,\mu}^{\lambda}$
\end{lemma}
\begin{proof}
By the same reasoning as the derivation of (\ref{eqn:Arabia}), we have $L_{\lambda,\mu}^{\lambda}=\xi_{\mu}(Z)|_{\lambda}$ and $O_{\lambda,\mu}^{\lambda}=\xi_{\mu}(Y)|_{\lambda}$. The lemma this holds since, by \cite[Theorem 3]{Ikeda.Naruse}, $\xi_{\mu}(Y)|_{\lambda}=2^{-\ell(\mu)} \xi_{\mu}(Z)|_{\lambda}$. 
\end{proof}

Assign weights to each box of the staircase $\rho_n$ as follows. For $Y$, the boxes on the main diagonal are assigned
weight $\gamma_1$. The boxes on the next diagonal are assigned $\gamma_2$, etc. For $Z$, the boxes on the main diagonal
are assigned $\alpha_1$ whereas the boxes on the second diagonal are assigned $2\alpha_2$, and the third diagonal $2\alpha_3$, etc.
Let $\beta_Y:=\rho_n\to \{\gamma_i\}$ and $\beta_Z:\rho_n\to \{\alpha_1,2\alpha_2,\ldots, 2\alpha_n\}$ be
these two assignments.

Thus, when $n=3$ the assignment is
\[\tableauL{\gamma_1 & \gamma_2 & \gamma_3\\  & \gamma_1 &  \gamma_2 \\   &  & \gamma_1} \text{\ \ \ (for $Y$) \ \ \ and  \ \ \ }
\tableauL{\alpha_1 & 2\alpha_2 & 2\alpha_3\\  & \alpha_1 &  2\alpha_2 \\   &  & \alpha_1}  \text{\ \ \ (for $Z$).}
\] 
For a straight shape $\lambda\subseteq \rho_n$, define 
\[{\tt wt}_{Y}(\lambda)=\sum_{x\in\lambda} \beta_Y(x).\]
For a skew shape $\nu/\lambda\subseteq \rho_n$, 
\[{\tt wt}_{Y}(\nu/\lambda):={\tt wt}_Y(\nu)-{\tt wt}_Y(\lambda).\]
Similarly, one defines ${\tt wt}_Z(\nu/\lambda)$.

Let $\lambda^+$ be $\lambda$ with a box added. Also let $\nu^{-}$ be $\nu$ with a box removed. We claim that
\begin{equation}
\label{eqn:Brec}
\sum_{\lambda^+} O_{\lambda^+,\mu}^{\nu}=O_{\lambda,\mu}^{\nu}{\tt wt}_Y(\nu/\lambda)+\sum_{\nu^-} O_{\lambda,\mu}^{\nu^-}.
\end{equation}
This is proved by the considering the associativity relation 
\[(\xi_{\lambda}(Y) \cdot \xi_{(1)}(Y))\cdot \xi_{\mu}(Y)=\xi_{\lambda}(Y) \cdot (\xi_{(1)}(Y)\cdot \xi_{\mu}(Y)),\]
and using the Pieri rule for $Y$: 
\begin{equation}
\label{eqn:PieriBbox}
\xi_{(1)}(Y) \cdot \xi_{\lambda}(Y) = {\tt wt}_Y(\lambda)\xi_{\lambda}(Y)+\sum_{\lambda^+} \xi_{\lambda^+}(Y).
\end{equation}
The proof of (\ref{eqn:PieriBbox}) can be obtained starting with the same reasoning as the derivation of (\ref{eqn:equivpieriabc}). Alternatively, it can be
deduced by specializing more general formulas such as C.~Lenart-A.~Postnikov's \cite[Corollary~1.2]{Lenart.Postnikov}.

Similarly, the Pieri rule for $Z$ reads
\[\xi_{(1)}(Z) \cdot \xi_{\lambda}(Z)={\tt wt}_Z(\lambda)\xi_{\lambda}(Z)+\sum_{\lambda^+} 2^{\ell(\lambda)+1-\ell(\lambda^+)} \xi_{\lambda^+}(Z).\]
Consequently, by the same reasoning we obtain
\begin{equation}
\label{eqn:Crec}
\sum_{\lambda^+} L_{\lambda^+,\mu}^{\nu}2^{\ell(\lambda)+1-\ell(\lambda^+)}=L_{\lambda,\mu}^{\nu}{\tt wt}_Z(\nu/\lambda)+\sum_{\nu^-} L_{\lambda,\mu}^{\nu^-}2^{\ell(\nu^-)+1-\ell(\nu)}.
\end{equation}

Now, to complete the proof by induction we start from (\ref{eqn:Crec}). This is an identity of polynomials and remains so after
the substitution $\alpha_1\mapsto 2\gamma_1$ and $\alpha_i\mapsto \gamma_i$ for $i>1$. That is,
\begin{equation}
\label{eqn:Crec'}
\sum_{\lambda^+} {\widetilde L}_{\lambda^+,\mu}^{\nu}2^{\ell(\lambda)+1-\ell(\lambda^+)}={\widetilde L}_{\lambda,\mu}^{\nu}{\widetilde {\tt wt}_Z(\nu/\lambda)}+\sum_{\nu^-} {\widetilde L}_{\lambda,\mu}^{\nu^-}2^{\ell(\nu^-)+1-\ell(\nu)},
\end{equation}
where ${\widetilde {\tt wt}_Z(\nu/\lambda)}$ is ${\tt wt}_Z(\nu/\lambda)$ with the same substitution. Note that 
\begin{equation}
\label{eqn:halfabc}
\frac{1}{2}{\widetilde {\tt wt}_Z(\nu/\lambda)}={\tt wt}_Y(\nu/\lambda).
\end{equation}
Now multiply both sides of (\ref{eqn:Crec'}) by $\frac{1}{2}\times 2^{\ell(\nu)-\ell(\lambda)-\ell(\mu)}$. This gives
\begin{equation}
\label{eqn:Crec''}
\sum_{\lambda^+} {\widetilde L}_{\lambda^+,\mu}^{\nu}2^{\ell(\nu)-\ell(\lambda^+)-\ell(\mu)}=2^{\ell(\nu)-\ell(\lambda)-\ell(\mu)}{\widetilde L}_{\lambda,\mu}^{\nu}{\tt wt}_Y(\nu/\lambda)+\sum_{\nu^-} {\widetilde L}_{\lambda,\mu}^{\nu^-}2^{\ell(\nu^-)-\ell(\lambda)-\ell(\mu)}.
\end{equation}
By induction,
\begin{equation}
\label{eqn:Crec'''}
\sum_{\lambda^+} {O}_{\lambda^+,\mu}^{\nu}=2^{\ell(\nu)-\ell(\lambda)-\ell(\mu)}{\widetilde L}_{\lambda,\mu}^{\nu}{\tt wt}_Y(\nu/\lambda)+\sum_{\nu^-} {O}_{\lambda,\mu}^{\nu^-}.
\end{equation}
 Comparing (\ref{eqn:Crec'''}) and (\ref{eqn:Brec}) we deduce that
 $O_{\lambda,\mu}^{\nu}=2^{\ell(\nu)-\ell(\lambda)-\ell(\mu)}{\widetilde L}_{\lambda,\mu}^{\nu}$, as needed.
\end{proof}

Turning to nonvanishing, clearly:
\begin{corollary}
\label{cor:Aug8abc}
$[\gamma_1^{i_1}\cdots \gamma_n^{i_n}]O_{\lambda,\mu}^{\nu}\neq 0 \iff [\alpha_1^{i_1}\cdots \alpha_n^{i_n}]L_{\lambda,\mu}^{\nu}\neq 0$; in particular
$O_{\lambda,\mu}^{\nu}\neq 0 \iff L_{\lambda,\mu}^{\nu}\neq 0$.
\end{corollary}

Moreover, C.~Monical \cite{Monical} gave a conjectural equivariant extension of Theorem~\ref{thm:PurbhooSottile}.
\begin{conjecture}[C.~Monical's Horn-type conjecture]\label{conj:Monical}
For $\lambda,\mu,\nu\subseteq \rho_n$ (and not a smaller staircase), $O_{\lambda,\mu}^{\nu^c}\neq 0$ if and only if
for $k=|\lambda|+|\mu|+|\nu|-\binom{n+1}{2}$,
\begin{itemize}
\item $k\geq 0$, and
\item for all $0<r<n$ and all $\alpha,\beta,\gamma\subseteq r\times (n-r)$ with $|\alpha|+|\beta|+|\gamma|=r(n-r)$ and
$c_{\alpha,\beta}^{\gamma^c}\neq 0$ we have $[\lambda]_{\alpha}+[\mu]_{\beta}+[\nu]_{\gamma}-k\leq \binom{n+1-r}{2}$. 
\end{itemize}
\end{conjecture}

In \emph{loc.~cit.}, C.~Monical reports checking this conjecture for all $\lambda,\mu,\nu\subseteq \rho_5$. Now,
from Corollary~\ref{cor:Aug8abc} we obtain:
\begin{corollary}[\emph{cf.} Conjecture~5.3 of \cite{Monical}]
C.~Monical's inequalities characterize $O_{\lambda,\mu}^{\nu}\neq 0$ if and only if they characterize $L_{\lambda,\mu}^{\nu}\neq 0$.
\end{corollary}

\section{Shifted edge labeled tableaux}\label{sec:shiftededge}

In this section, we define \emph{shifted edge labeled tableaux}. At present, we do not know a good theory when edge labels are permitted
on arbitrary horizontal edges. However, our central new idea is to \emph{restrict edge labels to diagonal boxes}.
This restriction gives rise to a combinatorial
rule which defines a commutative and (conjecturally) associative ring.

\subsection{Main definitions} If $\mu\subseteq \lambda$, then $\lambda/\mu$ is the skew-shape consisting of boxes
of $\lambda$ not in $\mu$. The boxes in matrix position $(i,i)$ are the \emph{diagonal boxes}. A \emph{diagonal edge} of $\lambda/\mu$ 
refers to the southern edge of a diagonal box of $\lambda$. If $\mu=\emptyset$, we call $\lambda=\lambda/\mu$ a \emph{straight shape}.

For example if $\lambda=(6,3,1)$ and $\mu=(3,1)$, the shape $\lambda/\mu$ consists of the six unmarked boxes shown below
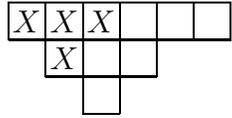
\[\tableauS{X & X &X  & \ & \ &\ \\  & X &\ &\ \\ & & \ }.\]
This has one diagonal box but three diagonal edges.

A \emph{shifted edge labeled tableau} of shape $\lambda/\mu$ is a filling of the boxes of $\lambda/\mu$ and southern edges of the diagonal boxes with the labels $[N]=\{1,2,3,\ldots,N\}$ such that:
\begin{itemize}
\item[(S1)] Every box of $\lambda/\mu$ is filled.
\item[(S2)] Each \emph{diagonal} edge contains a (possibly empty) subset of $[N]$.
\item[(S3)] $1,2,\ldots,N$ appears exactly once.
\item[(S4)] The labels strictly increase left to right along rows and top to bottom along columns. In particular, each label of a diagonal edge is strictly larger
than the box labels in the same column.
\end{itemize}
These conditions imply that $N\geq |\lambda/\mu|$. Let ${\sf eqShSYT}(\lambda/\mu,N)$ be the set of
all such tableaux. If we restrict to tableaux satisfying only (S1), (S3) and (S4), then $N=|\lambda/\mu|$ and we obtain the 
notion of shifted standard Young tableaux from Section~\ref{sec:PQstuff}. 

An \emph{inner corner} ${\sf c}$ of $\lambda/\mu$ is a maximally southeast box of $\mu$. For $T \in {\sf eqShSYT}(\lambda/\mu,N)$, we define a \emph{(shifted, edge labeled) jeu de taquin slide} ${\sf shEjdt}_{\sf c}(T)$, obtained as follows. Initially place $\bullet$ in ${\sf c}$, and apply one of the following \emph{slides}, depending on what $T$ looks like locally around ${\sf c}$:
\begin{itemize}
\item[(J1)] $\tableauS{\bullet & a\\ b}\mapsto \tableauS{b & a\\ \bullet }$ (if $b<a$, or $a$ does not exist)
\smallskip
\item[(J2)] $\tableauS{\bullet & a\\ b}\mapsto \tableauS{a & \bullet\\ b}$ (if $a<b$, or $b$ does not exist)
\smallskip
\item[(J3')] $\begin{picture}(40,30)\put(0,0){$\tableauS{\bullet & a}$}
\put(3,-5){$S$}
\end{picture} \mapsto
\begin{picture}(40,30)\put(0,0){$\tableauS{a & \bullet }$}
\put(3,-5){$S$}
\end{picture}$ (if ${\sf c}$ is a \emph{diagonal} box and $a<\min(S)$)

\item[(J4')] $\begin{picture}(40,30)\put(0,0){$\tableauS{\bullet & a}$}
\put(3,-5){$S$}
\end{picture} \mapsto
\begin{picture}(40,30)\put(0,0){$\tableauS{s & a }$}
\put(3,-5){$S'$}
\end{picture}$ (if ${\sf c}$ is a \emph{diagonal} box, $s:=\min(S)<a$ and $S':=S\setminus \{s\}$)
\end{itemize}
Repeat the above sliding procedure on the new box ${\sf c}'$ containing the new position $\bullet$ until $\bullet$ arrives at a box or diagonal edge ${\sf d}$ 
of $\lambda$ that has no labels immediately south or east of it. Then ${\sf shEjdt}_{\sf c}(T)$ is obtained by erasing $\bullet$.

A \emph{rectification} of $T\in {\sf eqShSYT}(\lambda/\mu,N)$ is defined as usual: Choose an inner corner ${\sf c}_0$ of $\lambda/\mu$ and compute
$T_1:={\sf shEjdt}_{{\sf c}_0}(T)$, which has shape $\lambda^{(1)}/\mu^{(1)}$. Now let ${\sf c}_1$ be an inner corner of $\lambda^{(1)}/\mu^{(1)}$ and compute
$T_2:={\sf shEjdt}_{{\sf c}_1}(T_1)$. Repeat $|\mu|$ times, arriving at a standard tableau of straight shape. Let ${\sf shEqRect}_{\{{\sf c}_i\}}(T)$ be this tableau.

In general, ${\sf shEqRect}$ is not independent of rectification order, when $N>|\lambda/\mu|$:

\begin{example}
The reader can check that  if one uses column rectification order
(picking the rightmost inner corner at each step) then
\[\begin{picture}(100,50) 
\put(5,35){$\tableauL{{\ }&{\ }&{\ }\\ & {\ }&{1} \\ & & {2}}$}
\put(50,-8){$3$}
\end{picture}\]
rectifies to $\tableauS{1 & 2 & 3}$ while row rectification (choosing the southmost inner corner at each step) gives $\tableauS{1 &2\\ &3}$. \qed
\end{example}

We will define ${\sf shEqRect}(T)$ to be the rectification under row rectification order.

\subsection{A (putative) commutative ring structure}
Let $S_{\mu}$ be the superstandard tableau of shifted shape $\mu$, which is obtained by filling the boxes of $\mu$ in English reading order
with $1,2,3,\ldots$. For example, 
\[S_{(5,3,1)}=\tableauS{1&2&3&4&5\\  &6&7&8\\ & & 9}.\]
Define
\[d_{\lambda,\mu}^{\nu}:=\#\{T\in {\sf eqShSYT}(\nu/\lambda,|\mu|): {\sf shEqRect}(T)=S_{\mu}\}.\]
Let
\[\Delta(\nu;\lambda,\mu):=|\lambda|+|\mu|-|\nu|
\text{ \ and $L(\nu;\lambda,\mu):=\ell(\lambda)+\ell(\mu)-\ell(\nu)$.}\]
Introduce an indeterminate $z$ and set 
\[D_{\lambda,\mu}^{\nu}:=2^{L(\nu;\lambda,\mu)-\Delta(\nu;\lambda,\mu)}z^{\Delta(\nu;\lambda,\mu)}d_{\lambda,\mu}^{\nu}.\]

 Next we define formal symbols $[\lambda]$ for each $\lambda\subseteq \rho_n$. Let $R_n$ be the free ${\mathbb Z}[z]$-module generated by these.
We declare a product structure on $R_n$ by
\[[\lambda]\star [\mu] = \sum_{\nu} D_{\lambda,\mu}^{\nu} [\nu].\]

\begin{table}[t]
\begin{tabular}{ |c|c|c| }
\hline
$\lambda$ & $\mu$ & $\lambda \star \mu $ \\  
\hline\hline
$[1]$ & $[1]$ & $z[1] + 2[2]$\\
$[1]$ & $[2]$ & $z[2] + [2,1] + 2[3]$\\
$[1]$ & $[2, 1]$ & $2z[2,1] + 2[3,1]$\\
$[1]$ & $[3]$ & $z[3] + [3,1] $\\
$[1]$ & $[3, 1]$ & $2z[3,1] + 2[3,2]$\\
$[1]$ & $[3, 2]$ & $2z[3,2] + [3,2,1]$\\
$[1]$ & $[3, 2, 1]$ & $3z[3,2,1] $\\
$[2]$ & $[2]$ & $z[2,1] + z[3] + 2[3,1]$\\
$[2]$ & $[2,1]$ & $ z^2[2,1] + 3z[3,1] + 2[3,2]$\\
$[2]$ & $[3]$ & $z[3,1] + [3,2]$\\
$[2]$ & $[3,1]$ & $z^2[3,1] + 3z[3,2] + [3,2,1]$ \\
$[2]$ & $[3, 2]$ & $z^2[3,2] + 2z[3,2,1]$\\
$[2]$ & $[3, 2, 1]$ & $3z^2[3,2,1]$\\
$[2,1]$ & $[2,1]$ & $z^3[2,1] + 3z^2[3,1] + 6z[3,2] $ \\ 
$[2,1]$ & $[3]$ & $z^2[3,1] + z[3,2] + [3,2,1]$\\
$[2,1]$ & $[3,1]$ & $z^3[3,1]  + 3z^2[3,2] + 3z[3,2,1]$  \\ 
$[2,1]$ & $[3,2]$ & $z^3[3,2] + 3z^2[3,2,1]$\\
$[2,1]$ & $[3,2,1]$ & $4z^3[3,2,1]$\\
$[3]$ & $[3]$ & $z[3,2]$\\
$[3]$ & $[3,1]$ & $ z^2[3,2] + z[3,2,1]$ \\
$[3]$ & $[3, 2]$ & $z^2[3,2,1]$\\
$[3]$ & $[3,2,1]$ & $z^3[3,2,1]$\\
$[3,1]$ & $[3,1]$ & $z^3[3,2] + 3z^2[3,2,1]$\\
$[3,1]$ & $[3,2]$ & $2z^3[3,2,1]$ \\
$[3,1]$ & $[3,2,1]$ & $2z^4[3,2,1]$\\
$[3,2]$ & $[3,2]$ & $z^4[3,2,1]$\\
$[3,2]$ & $[3,2,1]$ & $z^5[3,2,1]$\\
$[3,2,1]$ & $[3,2,1]$ & $z^6[3,2,1]$\\
\hline
\end{tabular}
\caption{Table of products for $n=3$}
\end{table}

While positivity of $D_{\lambda,\mu}^{\nu}$ is immediate from the definition, the following is not:
\begin{conjecture}
\label{conj:integral}
$D_{\lambda,\mu}^{\nu} \in \mathbb{Z}[z]$.
\end{conjecture}

\begin{example}
Suppose that $\lambda=(2,1), \mu=(3,1), \nu=(3,1)$. Then $\Delta(\nu;\lambda,\mu)=3+4-4=3$, 
$L(\nu;\lambda,\mu)=2+2-2=2$. Also, $d_{\lambda,\mu}^{\nu} = 2$ because the following are the only $2$ shifted edge labeled tableaux which rectify to $S_{\mu}$.

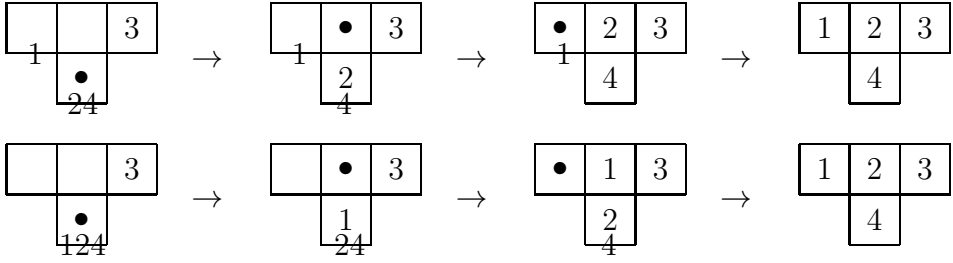
\[\begin{picture}(300,50)
\put(-50,35){$\tableauL{{\ }&{\ }&{3}\\ & {\bullet}}$}
\put(-42,31){$1$}
\put(-27,12){$24$}
\put(20,30){$\rightarrow$}
\put(50,35){$\tableauL{{\ }&{\bullet}&{3}\\ & {2 }}$}
\put(58,31){$1$}
\put(75,12){$4$}
\put(120,30){$\rightarrow$}
\put(150,35){$\tableauL{{\bullet }&{2}&{3}\\ & {4}}$}
\put(158,31){$1$}
\put(220,30){$\rightarrow$}
\put(250,35){$\tableauL{{1}&{2}&{3}\\ & {4}}$}

\end{picture}\]
\[\begin{picture}(300,50)
\put(-50,35){$\tableauL{{\ }&{\ }&{3}\\ & {\bullet}}$}
\put(-30,12){$124$}
\put(20,30){$\rightarrow$}
\put(50,35){$\tableauL{{\ }&{\bullet}&{3}\\ & {1}}$}
\put(74,12){$24$}
\put(120,30){$\rightarrow$}
\put(150,35){$\tableauL{{\bullet }&{1}&{3}\\ & {2}}$}
\put(175,12){$4$}
\put(220,30){$\rightarrow$}
\put(250,35){$\tableauL{{1}&{2}&{3}\\ & {4}}$}
\end{picture}\]
Thus $D_{\lambda,\mu}^{\nu}=2^{2-3}\times z^{3}\times 2 = z^3$. \qed
\end{example}

In the previous example, $2^{L(\nu;\lambda,\mu)-\Delta(\nu;\lambda,\mu)}=2^{-1}$. Further, in the tableaux counting $d_{\lambda,\mu}^{\nu}$, only the edge labels differed. In the case that $L(\nu;\lambda,\mu)-\Delta(\nu;\lambda,\mu)=-k<0$ one might wonder if the tableaux counting $d_{\lambda,\mu}^{\nu}$, namely
\[F(\lambda,\mu;\nu) = \{T\in {\sf shEqSYT}(\nu/\lambda,|\mu|): {\sf shEqRect}(T)=S_{\mu}\}\] 
can be sorted into equivalence classes of size $2^k$ by ignoring edge labels. The following example shows this is false in general:

\begin{example} Suppose that $\lambda=(3), \mu=(3,2,1), \nu=(4,2,1)$. Then $\Delta(\nu;\lambda,\mu)=3+6-7=2$ and 
$L(\nu;\lambda,\mu)=1+3-3=1$, so $k=1$. Below a $T\in F_{(3),(3,2,1);(4,2,1)}$. Any $
T'\in {\sf eqShSYT}((4,2,1)/(3),6)$ formed by moving the edge labels of $T$ is not in $F_{(3),(3,2,1);(4,2,1)}$.
        
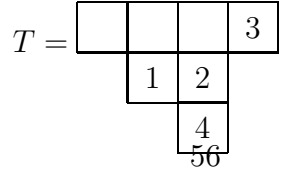
\[\begin{picture}(100,60)
\put(0,45){$T=\tableauL{{\ }&{\ }&{\ }&{3}\\ & {1}& {2}\\ & & {4}}$}
\put(67,2){$56$}
\end{picture}\]
\end{example}

While Conjecture~\ref{conj:integral} is a purely combinatorial question, it
would also follow from a conjectural connection to equivariant 
Schubert calculus, through work of D.~Anderson-W.~Fulton presented in Section~\ref{sec:AndersonFultonstuff}.

The next result gives a further consistency check of our combinatorics. It was suggested by
 H.~Thomas (private communication):

\begin{theorem}
\label{theorem:commutative}
$R_n$ is commutative, \emph{i.e.}, $D_{\lambda,\mu}^{\nu}=D_{\mu,\lambda}^{\nu}$.
\end{theorem}
\begin{proof}
The argument is based on a variation of S.~Fomin's \emph{growth diagram} formulation of \emph{jeu de taquin}; see, \emph{e.g.}, 
\cite[Appendix~1]{ECII}.

Given a tableau $T\in {\sf eqShSYT}(\lambda/\theta,n)$, define the corresponding \emph{e-partition} (e for ``edge'') to be ${\sf epart}(T):=(\lambda_{1}^{i_1},\lambda_{2}^{i_2}, \ldots )$ where $i_k = $number of edge labels on the $k^{th}$ diagonal edge. 
Each such $T$ can be encoded as a sequence of e-partitions starting with the (usual) partition $\theta = (\theta_{1}^{0}, \theta_{2}^{0}, \ldots)$: If $1$ appears in a box in row $i$ then the next e-partition has an extra box in this position, \emph{i.e.}, we replace $\theta_{i}^{0}$ with ${(\theta_i+1)}^{0}$. Otherwise $1$ appears on the edge of a diagonal box in row $i$, in which case, the one-larger e-partition has $\theta_{i}^{0}$ replaced by $\theta_i^{1}$. Repeat this process by looking at the position of $2$ in $T$ \emph{etc.} 
Evidently, such an encoding of $T$ is unique.

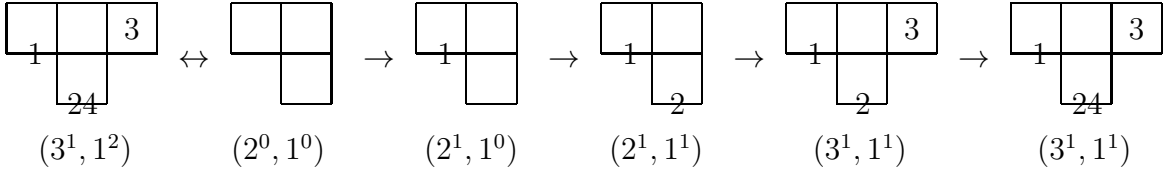
\begin{example}
\[\begin{picture}(300,50)
\put(-70,35){$\tableauL{{\ }&{\ }&{3}\\ & {\ }}$}
\put(-62,31){$1$}
\put(-47,12){$24$}
\put(-58,-5){$(3^1,1^2)$}

\put(-5,30){$\leftrightarrow$}
\put(15,35){$\tableauL{{\ }&{\ }\\ & {\ }}$}
\put(15,-5){$(2^0,1^0)$}

\put(65,30){$\rightarrow$}
\put(85,35){$\tableauL{{\  }&{\ }\\ & {\ }}$}
\put(93,31){$1$}
\put(88,-5){$(2^1,1^0)$}

\put(135,30){$\rightarrow$}
\put(155,35){$\tableauL{{\ }&{\ }\\ & {\ }}$}
\put(163,31){$1$}
\put(181,12){$2$}
\put(158,-5){$(2^1,1^1)$}

\put(205,30){$\rightarrow$}
\put(225,35){$\tableauL{{\ }&{\ }&{3}\\ & {\ }}$}
\put(233,31){$1$}
\put(251,12){$2$}
\put(235,-5){$(3^1,1^1)$}

\put(290,30){$\rightarrow$}
\put(310,35){$\tableauL{{\ }&{\ }&{3}\\ & {\ }}$}
\put(318,31){$1$}
\put(333,12){$24$}
\put(320,-5){$(3^1,1^1)$}

\end{picture}\]
\end{example}

Rectifying the left tableau above:
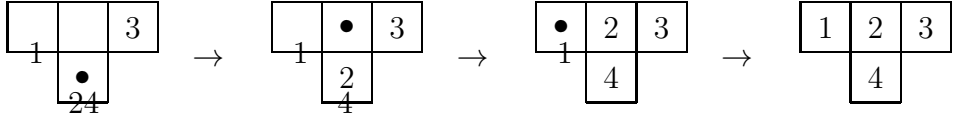
\[\begin{picture}(300,60)
\put(-50,35){$\tableauL{{\ }&{\ }&{3}\\ & {\bullet}}$}
\put(-42,31){$1$}
\put(-27,12){$24$}
\put(20,30){$\rightarrow$}
\put(50,35){$\tableauL{{\ }&{\bullet}&{3}\\ & {2 }}$}
\put(58,31){$1$}
\put(75,12){$4$}
\put(120,30){$\rightarrow$}
\put(150,35){$\tableauL{{\bullet }&{2}&{3}\\ & {4}}$}
\put(158,31){$1$}
\put(220,30){$\rightarrow$}
\put(250,35){$\tableauL{{1}&{2}&{3}\\ & {4}}$}
\end{picture} \]
Each of these four tableaux also has an associated sequence of e-partitions. Place these atop of one another as below.
The result is a \emph{tableau rectification diagram}:
\bgroup
\def\arraystretch{1.5}
\begin{center}
\begin{tabular}{ |c|c|c|c|c| } 
 \hline
 $(2^0,1^0)$ & $(2^1,1^0)$ & $(2^1,1^1)$ & $(3^1,1^1)$ & $(3^1,1^2)$ \\ 
 \hline
 $(2^0)$ & $(2^1)$ & $(2^1,1^0)$ & $(3^1,1^0)$ & $(3^1,1^1)$ \\ 
 \hline 
 $(1^0)$ & $(1^1)$ & $(2^1)$ & $(3^1)$ & $(3^1,1^0)$ \\ 
 \hline
 $\emptyset$ & $(1^0)$ & $(2^0)$ & $(3^0)$ & $(3^0,1^0)$ \\ 
 \hline
\end{tabular}
\end{center}
\egroup

Given two e-partitions $\lambda=(\lambda_{1}^{i_1}, \lambda_{2}^{i_2}, \ldots)$ and $\mu = (\mu_{1}^{j_1}, \mu_{2}^{j_2}, \ldots)$, we will say $\mu$ \emph{covers} $\lambda$ if: 
\begin{itemize}
\item[(i)] there exists unique $m$ such that $\lambda_m + 1= \mu_m$ and $\lambda_k= \mu_k$ for $k\neq m$, and $i_k=j_k$ for all $k$; or

\item[(ii)] $\lambda_k=\mu_k$ for all $k$ and there exists a unique $m$ such that $i_m+1=j_m$ and $i_k= j_k$ for $k\neq m$.
\end{itemize}

In the case that $\mu$ covers $\lambda$ we define $\mu/\lambda$ to be the extra box added in row $m$ (if in cases (i) above), or
the $m^{th}$ diagonal edge (in case (ii)). 
If $x$ is a diagonal edge, \emph{define} ${\sf shEjdt}_x(T)=T$. 
For two e-partitions, $\lambda=(\lambda_{1}^{i_1}, \lambda_{2}^{i_2}, \ldots)$ and 
$\mu = (\mu_{1}^{j_1}, \mu_{2}^{j_2}, \ldots)$, 
let 
\[\lambda \vee \mu = 
({\sf max}{\{\lambda_1,\mu_1\}}^{i_1+j_1}, 
{\sf max} {\{\lambda_2,\mu_2\}}^{i_2+j_2}, \ldots).\]

Consider the following local conditions on any $2 \times 2$ subsquare $\tableauS{{\alpha}&{\beta}\\{\gamma}&{\delta}}$ on a grid of e-partitions:

\begin{itemize}
\item[(G1)] Each e-partition covers the e-partition immediately to its left or below.
\item[(G2)] $\delta = \gamma \vee {\sf epart}({\sf shEjdt}_{\alpha/\gamma}(T))$, where $T$ is the filling of $\beta/\alpha$ by $1$. Similarly $\alpha = \gamma \vee {\sf epart}({\sf shEjdt}_{\delta/\gamma}(T))$ where $T$ is the filling of $\beta/\delta$ by $1$.	
\end{itemize}
 
Call any rectangular table of e-partitions satisfying (G1) and (G2) a \emph{growth diagram}. By the symmetry in the
definition of  (G1) and (G2), if $\mathcal{G}$ is a growth diagram, then so is $\mathcal{G}$ reflected about its antidiagonal. The following is
straightforward from the definitions:

\begin{claim}
\label{claim:growth1}
If $\tableauS{{\alpha}&{\beta}\\{\gamma}&{\delta}}$ is a $2 \times 2$ square in the tableau rectification diagram, then (G1) and (G2) hold. 
\end{claim}
\begin{proof}
Fix any two rows of the tableau rectification diagram; call this $2\times (n+1)$ subdiagram ${\mathcal R}$. The higher of the two row corresponds to some shifted edge labeled tableau $U$ and
the other row corresponds to ${\sf shEjdt}_{\sf c}(U)$ where ${\sf c}$ is a box (determined by the shapes in the leftmost column). Now $U$ is filled by $1,2,\ldots,n$. Notice that if we consider the submatrix ${\mathcal R}'$ of ${\mathcal R}$ consisting of the leftmost $k+1$ columns, then ${\mathcal R}'$ corresponds to the
computation of ${\sf shEjdt}_{\sf c}(U')$ where $U'$ is $U$ with labels $k+1,k+2,\ldots,n$ removed. The upshot is that it suffices to prove the claim for the rightmost
$2\times 2$ square in ${\mathcal R}$, which we will label with shapes $\tableauS{{\alpha}&{\beta}\\{\gamma}&{\delta}}$. Let $U_{\beta}$ be the tableau
associated to the chain of e-partitions ending at $\beta$. Similarly define $U_{\alpha}$. As well we have
\begin{equation}
\label{eqn:Aug19abc}
U_{\gamma}={\sf shEjdt}_{\sf c}(U_{\alpha}) 
\end{equation}
and
\begin{equation}
\label{eqn:Aug19xyz}
U_{\delta}={\sf shEjdt}_{\sf c}(U_{\beta}).
\end{equation}
Insofar as (G1) is concerned, it is obvious that $\beta$ covers $\alpha$ and $\delta$ covers $\gamma$. That $\beta$ covers $\delta$ follows from
(\ref{eqn:Aug19xyz}). Similarly, $\alpha$ covers $\gamma$ because of (\ref{eqn:Aug19abc}).

Now we turn to the proof of the first sentence of (G2). 
Suppose that $S\in {\sf eqShSYT}(\lambda/\theta,n)$. Define $\bar{S}$ to be the tableau obtained by forgetting the entries $1,2,\ldots n-1$ in $S$ and replacing the $n$ by $1$. Also define $\tilde{S}$ to be the tableau obtained by forgetting the entry $n$ in $S$. Then, it is clear that
\begin{equation}\label{eqn:epart}
{\sf epart}(S)= {\sf epart}(\tilde{S})\vee {\sf epart}(\bar{S})
\end{equation}
By definition, $\tilde{U_{\delta}}=U_{\gamma}$. This combined with (\ref{eqn:epart}) applied to $S=U_{\delta}$ shows that to prove the claim it suffices to show that ${\sf shEjdt}_{\alpha/\gamma}(T)=\bar{U_{\delta}}$.

\noindent {\sf Case 1:} ($\alpha$ covers $\gamma$ by (i))
In the computation of ${\sf shEjdt}_{\sf c}(U_{\alpha})$, the
$\bullet$ (that starts at ${\sf c}$) arrives at the outer corner \emph{box} $\alpha/\gamma$.
By definition, $U_\beta$ contains $n$ at $\beta/\alpha$. Therefore, if $\alpha/\gamma$ is not adjacent to $\beta/\alpha$, clearly 
$U_{\delta}$ is $U_{\gamma}$ with $n$ adjoined at $\beta/\alpha$. Thus, ${\sf shEjdt}_{\alpha/\gamma}(T) = \bar{U_{\delta}}$
as desired.

Otherwise $\alpha/\gamma$ is adjacent to $\beta/\alpha$. Then by the definition of ${\sf shEjdt}$, the position of $1$ in ${\sf shEjdt}_{\alpha/\gamma}(T)$ is the same as the position of $n$ in $U_{\delta}$. Thus, ${\sf shEjdt}_{\alpha/\gamma}(T) = \bar{U_{\delta}}$.

\noindent {\sf Case 2:} ($\alpha$ covers $\gamma$ by (ii))
Then in the computation of ${\sf shEjdt}_{\sf c}(U_{\alpha})$, the
$\bullet$ must have arrived at a diagonal box, and $k(<n)$ is the smallest edge label of this same box, resulting in a (J4') slide. Now regardless of where $n$ is placed
in $U_{\beta}$, it is clear that $U_{\delta}$ is $U_{\gamma}$ with $n$ adjoined in the same place as $n$'s place in $U_\beta$, \emph{i.e.}, $\beta/\alpha$. In other words, ${\sf shEjdt}_{\alpha/\gamma}(T) = \bar{U_{\delta}}$.

Proof of the second sentence of (G2):
For a pair of e-partitions $\lambda$, $\mu$ with $\mu$ covering $\lambda$, define $U_{\mu/\lambda}$ to be the tableau with $1$ placed in the location $\mu/\lambda$ in $\mu$. Clearly, $\alpha={\sf epart}(U_{\alpha})= {\sf epart}(U_{\gamma}) \vee {\sf epart}(U_{\alpha/\gamma})$ Thus, it suffices to show 
\begin{equation}\label{eqn:growAlpha}
    {\sf shEjdt}_{\delta/\gamma}(T)=U_{\alpha/\gamma}.
\end{equation}

\noindent {\sf Case 1:}($\delta$ covers $\gamma$ by (i))
 Then in $U_{\delta}$, $n$ occupies box $\delta/\gamma$.
If $\delta/\gamma$ is not adjacent to $\beta/\delta$, it is clear that $U_{\beta}$ and $U_{\delta}$ have $n$ in the same place, \emph{i.e.}, $\delta/\gamma$. Thus $\beta/\delta$ and $\alpha/\gamma$ are the same box or edge position. 
Thus Equation (\ref{eqn:growAlpha}) follows. 

Otherwise $\delta/\gamma$ is adjacent to $\beta/\delta$. Then, 
\begin{equation}
\label{eqn:Aug22abc}
{\sf shEjdt}_{\delta/\gamma}(T) = U_{\delta/\gamma}.
\end{equation} 
By the definition of ${\sf shEjdt}$, it follows that in the computation of ${\sf shEjdt}_{\sf c}(U_{\alpha})$, the $\bullet$ arrived at an outer corner $\alpha/\gamma$. This combined with the fact that $\delta/\gamma$ is adjacent to $\beta/\delta$, we conclude that $\alpha/\gamma$ is a box that is in the same
position as the box 
$\delta/\gamma$. Now (\ref{eqn:Aug22abc}) is precisely (\ref{eqn:growAlpha}).

\noindent {\sf Case 2:} ($\delta$ covers $\gamma$ by (ii))
Then $\delta/\gamma$ is an edge, so $n$ occupies an edge in $U_{\delta}$. Thus in ${\sf shEjdt}_{\sf c}(U_{\beta})$, $n$ is never moved.
So, $\beta/\delta=\alpha/\gamma$. Since $\delta/\gamma$ is a diagonal edge, ${\sf shEjdt}_{\delta/\gamma}(T)= T:=  U_{\beta/\delta}= U_{\alpha/\gamma}$.
\end{proof}

Let ${\sf Growth}(\lambda,\mu;\nu)$ be the set of growth diagrams such that:
\begin{itemize}
\item the leftmost column encodes the superstandard tableau of shape $\lambda$;
\item the bottom-most row encodes the superstandard tableau of shape $\mu$;
\item the shape of the e-partition in the top right corner is $\nu$.
\end{itemize}

\begin{claim}
\label{claim:growth2}
$\#{\sf Growth}(\lambda,\mu;\nu)=\#F(\lambda,\mu;\nu)$
\end{claim}
\begin{proof}
Given $T\in F(\lambda,\mu;\nu)$, form the tableau rectification diagram ${\mathcal G}(T)$. Notice that since we are using row rectification order, the
left side of the diagram will be the sequence for $S_{\lambda}$. Since $T$ is assumed to rectify to $S_{\mu}$, the bottom row of the
diagram will be the sequence  for $S_{\mu}$. Hence by Claim~\ref{claim:growth1}, ${\mathcal G}(T)\in{\sf Growth}(\lambda,\mu;\nu)$, and thus
$T\mapsto {\mathcal G}(T)$ is an injection implying $\#F(\lambda,\mu;\nu)\leq \#{\sf Growth}(\lambda,\mu;\nu)$. 

For the reverse inequality,
given any $\mathcal{G} \in {\sf Growth}(\lambda,\mu;\nu)$, by (G1), the top row defines $T({\mathcal G})\in {\sf eqShSYT}(\nu/\lambda,|\mu|)$. Then $T(\mathcal G)$ has a 
tableau rectification diagram ${\mathcal G}'$. By Claim~\ref{claim:growth1}, ${\mathcal G}'$ is uniquely determined by its left and top borders together with (G2). Thus, since ${\mathcal G}$ and ${\mathcal G}'$ share the same left and top borders and both satisfy (G2),
${\mathcal G}={\mathcal G}'$. In particular, $T({\mathcal G})\in F(\lambda,\mu;\nu)$. Thus, ${\mathcal G}\mapsto T(\mathcal G)$ is an 
injection proving $\#{\sf Growth}(\lambda,\mu;\nu)\leq \#F(\lambda,\mu;\nu)$.
\end{proof}

To conclude, we must show that $d_{\lambda,\mu}^{\nu} = d_{\mu,\lambda}^{\nu}$. Since 
\[d_{\lambda,\mu}^{\nu}=\#F(\lambda,\mu;\nu)=\#{\sf Growth} (\lambda,\mu ;\nu),\] 
it suffices to show that $\#{\sf Growth}(\lambda,\mu;\nu) = \#{\sf Growth}(\mu,\lambda;\nu)$. Reflecting along the antidiagonal defines a bijection between ${\sf Growth}(\lambda,\mu;\nu)$ and ${\sf Growth}(\mu,\lambda;\nu)$.
\end{proof}

\begin{example}
Under column rectification, Theorem~\ref{theorem:commutative} is false. 
Suppose $\lambda=(4,3)$, $\mu=(3,2,1)$ and $\nu=(4,3,2,1)$. The number of tableaux of shape $\nu/\lambda$ that column rectify to $S_{\mu}$ is $20$ while the number of those with shape $\nu/\mu$ column rectifying to $S_{\lambda}$ is $16$.\qed
\end{example}

\begin{conjecture}\label{conj:associative}
$(R_n,\star)$ is an associative ring.
\end{conjecture}

Additional support for Conjecture~\ref{conj:associative} comes from a conjectural connection to a commutative, associative ring studied by
D.~Anderson-W.~Fulton, as described in the next section.

\section{Conjectural connection to work of W.~Fulton-D.~Anderson and equivariant Schubert calculus}\label{sec:AndersonFultonstuff}

\subsection{Results of W.~Fulton-D.~Anderson}
For a strict shape
$\lambda\subseteq \rho_n$, let $\sigma_{\lambda}={\sf Pf}(c_{\lambda_i,\lambda_j})$ where 
\[c_{p,q}=\displaystyle\sum_{0\leq a\leq b \leq q}(-1)^b\left(\binom{b}{a}+\binom{b-1}{a}\right)z^a c_{p+b-a} c_{q-b}.\] 
If $\ell = \ell(\lambda)$ is odd, define $\lambda_{\ell+1}=0$ so that the matrix becomes even ordered.

Recently, D.~Anderson-W.~Fulton have studied a ${\mathbb Z}[z]$-algebra 
\[{\FAring}=\mathbb{Z}[z,c_1,c_2,\ldots]/(c_{p,p}=0, \forall p>0)\]
 and shown it has a basis over ${\mathbb Z}[z]$ of $\sigma_{\lambda}$. Define structure constants by
\[\sigma_{\lambda}\cdot \sigma_{\mu} =\sum_{\nu\subseteq \rho_n} \FA_{\lambda,\mu}^{\nu}\sigma_{\nu}.\]
Also let
\[{\mathfrak d}_{\lambda,\mu}^{\nu}:=\frac{\FA_{\lambda,\mu}^{\lambda}}{2^{L(\lambda;\lambda,\mu)-\Delta(\lambda;\lambda,\mu)}z^{\Delta(\lambda;\lambda,\mu)}}.\]
\begin{conjecture}
\label{conj:C=D}
There is a ring isomorphism $\phi:R_n\to \FAring$ that sends $[\lambda]\mapsto \sigma_{\lambda}$. Therefore,
$\FA_{\lambda,\mu}^{\nu}=D_{\lambda,\mu}^{\nu}$ (equivalently ${\mathfrak d}_{\lambda,\mu}^{\nu}=d_{\lambda,\mu}^{\nu}$).
\end{conjecture}

We have exhaustively checked Conjecture~\ref{conj:C=D} for all $n\leq 4$ and many $n=5$ cases. 

D.~Anderson-W.~Fulton (private communication) connected the above ring to equivariant Schubert calculus of $Z$. That is,
\begin{equation}
\label{eqn:FultonAndersonconnect}
L_{\lambda,\mu}^{\nu}(\alpha_1\mapsto z, \alpha_2\mapsto 0,\ldots,\alpha_n\mapsto 0)={\mathfrak D}_{\lambda,\mu}^{\nu}.
\end{equation}

\begin{proposition}
Conjecture~\ref{conj:C=D}$\implies$ Conjecture~\ref{conj:integral}.
\end{proposition}
\begin{proof}
By Theorem~\ref{thm:Graham},
$L_{\lambda,\mu}^{\nu}$ is a nonnegative integer polynomial in 
$\alpha_1,\ldots,\alpha_n$, the simples of the type $C$ root system. Now apply (\ref{eqn:FultonAndersonconnect}).
\end{proof}

In turn, Conjecture~\ref{conj:C=D} should follow from a proof  that $D_{\lambda,(p)}^{\nu}=\FA_{\lambda,(p)}^{\nu}$,
together with Conjecture~\ref{conj:associative}, by a variation of the ``associativity argument'' of Section~\ref{sec:Young}.

\subsection{Two numerologically nice  cases of Conjecture~\ref{conj:C=D}} 

\begin{theorem}
$d_{\lambda,(p)}^{\lambda}=\binom{\ell(\lambda)}{p}2^{p-1}={\mathfrak d}_{\lambda,(p)}^{\lambda}$.
\end{theorem}
\begin{proof}
We will use some results of T.~Ikeda-H.~Naruse \cite{Ikeda.Naruse} that we now recall.
For a strict partition $\lambda = (\lambda_1 > \ldots > \lambda_r>0)$,  let $D_{\lambda}$ denote the associated shifted shape. Explicitly, 
\[D_{\lambda} = \{(i,j)\in \mathbb{Z}^2| 1\leq i\leq r, i\leq j <\lambda_i +i \}.\] 
For instance, $D_{(3,1)}=\tableauS{{\ }&{\  }&{\ }\\ & {\ }}$. Given an arbitrary subset $C\subset D_{\lambda}$, if a box $x\in C$ satisfies either of the following conditions:
\begin{enumerate}
\item[(I)] $x=(i,i)$ and $(i,i+1), (i+1,i+1) \in D_{\lambda} \setminus C$
\item[(II)] $x=(i,j), j\neq i$ and $(i+1,i), (i,i+1), (i+1,i+1) \in D_{\lambda} \setminus C$
\end{enumerate}
then set 
$C'=C\cup \{ x+(1,1)\} \backslash \{x\}$. 
The procedure $C\rightarrow C'$ is called an \emph{elementary excitation occuring} at $x \in C$. Any subset $C'\subset D_{\lambda}$ obtained from $C$ by an application of successive elementary excitations is called 
an \emph{excited Young diagram} (EYD) of $C$. 
Denote by $\mathcal{E}_{\lambda}(\mu)$ the set of all EYDs of $D_{\mu}$ contained in $D_{\lambda}$.

\begin{example}\label{ex:Eyd1}
Suppose $\lambda=(4,2,1), \mu=(2)$, then $\mathcal{E}_{\lambda}(\mu)$ consists of the following EYDs
\begin{gather*}
\tableauS{{+}&{+}&{\ }&{\ }\\  & {\ }&{\ }\\  & & {\ }}
\qquad
\tableauS{{+ }&{\  }&{\ }&{\ }\\  & {\ }&{+}\\  & & {\ }}
\qquad
\tableauS{{\ }&{\  }&{\ }&{\ }\\  & {+ }&{+ }\\  & & {\ }}.
\end{gather*}
\end{example}

\begin{lemma}
\label{lemma:Aug10abc}
$\FA_{\lambda,\mu}^{\lambda} = \#\mathcal{E}_{\rho_{\ell(\lambda)}}(\mu)\times z^{|\mu|}$.
\end{lemma}
\begin{proof}
This follows from \cite[Theorem~3]{Ikeda.Naruse} which gives a formula for $\xi_{\lambda}(Z)|_{\mu}=L_{\lambda,\mu}^{\mu}$
 in terms of EYDs, combined with (\ref{eqn:FultonAndersonconnect}). We omit the details, which amount mostly to translating
 from the appropriate Weyl group elements to the associated partitions.
\end{proof}

By Lemma~\ref{lemma:Aug10abc}, 
\begin{equation}
\label{eqn:loc_eyd} 
{\mathfrak{d}}_{\lambda,\mu}^{\lambda}= \frac{\FA_{\lambda,\mu}^{\lambda}}{2^{L(\lambda;\lambda,\mu)-\Delta(\lambda;\lambda,\mu)}z^{\Delta(\lambda;\lambda,\mu)}}=
\frac{ \#\mathcal{E}_{\rho_{\ell(\lambda)}}(\mu) \times z^{|\mu|}}{ 2^{\ell(\mu)-|\mu|}z^{|\mu|}} = 
 \#\mathcal{E}_{\rho_{\ell(\lambda)}}(\mu)\times 2^{|\mu|-\ell(\mu)} .
\end{equation}

 As all boxes in $D_{(p)}$ are in distinct columns, they stay in distinct columns even after the application of the excitation moves (I) and (II). Therefore, $\mathcal{E}_{\rho_{\ell(\lambda)}}((p))$ contains at most $\binom{\ell(\lambda)}{p}$ elements, since, there are $\ell(\lambda)$ columns in $D_{\rho_{\ell(\lambda)}}$. It is not
 hard to see from (I) and (II) this upper bound is an equality. This proves the second equality of the theorem.

Let $N=\ell(\lambda)$; we now prove the first equality of the theorem statement by induction on $N+p$. When $N+p\leq 1$ the claim is obvious. When $N+p=2$, there is one case, namely, $\lambda=(1)$, $p=1$ and $d_{\lambda,(p)}^{\lambda}=1=\binom{1}{1}2^{1-1}$, as desired. Now suppose $N+p=k>2$ and the claim holds for smaller
$N+p$. 

Let $F(\lambda,(p);\lambda)$ be the tableaux enumerated by $d_{\lambda,(p)}^{\lambda}$. If $T\in F(\lambda,(p);\lambda)$ we say that a label $q$ \emph{appears in} row $r$ if $q$ is an edge label on the southern edge of the diagonal box in row $r$.
Let $\overline{T}$ be $T$ with the first row removed, this is of shape $\overline\lambda$. There are three disjoint cases that $T$ can
fall into:

\begin{enumerate}

\item ($1$ does not appear in row $1$ of $T$ and $\overline{T}\in F({\overline\lambda,(p);\overline\lambda})$:
Then there are $d_{\overline\lambda,(p)}^{\overline\lambda}$ many such choices; this equals  $\binom{N-1}{p}2^{p-1}$, by induction.

\item ($1$ does not appear in row $1$ of $T$ and $\overline{T}\not\in F({\overline\lambda,(p);\overline\lambda})$): Then it is
straightforward to check (from the assumption that $T\in F(\lambda,(p);\lambda)$) that ${\overline T}$ row rectifies to ${\overline S}$ of shape $(p-1)$ where the first row consists of 
box labels $1,3,4,\ldots, p-1$ and has a $2$ in the south edge of the first box. Notice that the choices for ${\overline T}$ are in bijection
with $F({\overline\lambda},(p-1);{\overline\lambda})$ where the map is to remove the edge label $2$ and shift the labels $3,4,5\ldots,p$ down by one. This, combined with induction asserts that there are $d_{\overline\lambda,(p-1)}^{\overline\lambda}=
\binom{N-1}{p-1}2^{p-2}$ many choices.

\item ($1$ appears in row $1$): No other label appears in row $1$ of $T$. Let $U$ be $\overline T$ with every entry decremented by one.
It is straightforward that $U\in 
F({\overline\lambda},(p-1);\overline\lambda)$, and that the map $T\mapsto U$ is bijective. Thus, there are
$d_{\overline\lambda,(p-1)}^{\overline\lambda}=\binom{N-1}{p-1}2^{p-2}$ many tableaux in this case, by induction. 
\end{enumerate}

By Pascal's identity, 
\[\binom{N}{p}2^{p-1} = \binom{N-1}{p}2^{p-1} + \binom{N-1}{p-1}2^{p-2} + \binom{N-1}{p-1}2^{p-2}.\]
This, combined with cases (1)-(3), completes the induction.
\end{proof}

\begin{theorem}\label{thm:locCoeff}
$d_{\rho_n,\rho_n}^{\rho_n}=2^{\binom{n}{2}}={\mathfrak d}_{\rho_n,\rho_n}^{\rho_n}$
\end{theorem}
\noindent
\emph{Proof sketch:}
In the special case when $\ell(\lambda) = \ell(\mu)$, it is easy to observe that $|\mathcal{E}_{\mu}^{\lambda}| = 1$. Thus in this case,
\begin{equation}
\label{eqn:localization}
{\mathfrak d}_{\lambda,\mu}^{\lambda}= 2^{|\mu|-\ell(\mu)} 
\end{equation}

Further when $\mu=\lambda=\rho_n$, 
\[{\mathfrak d}_{\rho_n,\rho_n}^{\rho_n}= 2^{|\mu|-\ell(\mu)}=2^{\binom{n}{2}}.\]
This proves the rightmost equality.

For the remaining equality, consider $T\in {\sf shEqSYT}(\rho_n/\rho_n,N)$ where $N=|\rho_n|=\binom{n+1}{2}$.
For $1\leq i\leq n$, let \[E_i(T)=\{k \ | \ k \mbox{ lies on the } i \mbox{th diagonal edge of } T\}.\]

For $T\in {\sf shSYT}(\nu/\lambda)$, let $T(i,j)$ be the entry in box $(i,j)$ (in matrix coordinates). 
Define $U_n\in{\sf shEqSYT}(\rho_n/\rho_n,N)$ by the requirement that
$E_i(U_n)=\bigcup_{r=1}^i S_{\rho_n}(r,i)$. That is, the labels on the $i$th diagonal edge of $U_n$
are precisely the labels appearing in column $i$ of $S_{\rho_n}$.

For $T\in {\sf shEqSYT}(\rho_n/\rho_n,N)$ and $I \subseteq E_i(T)$ for some $i\in[n-1]$, define the \emph{$I$-slide} of $T$, 
${\sf Sl}_{I}(T)\in{\sf shEqSYT}(\rho_n/\rho_n,N)$, by 
  \[E_k({\sf Sl}_{I}(T)):=\begin{cases}
	E_k(T) &\text{if } k\in [n]\setminus\{i,i+1\},\\
    E_k(T)\setminus {I} &\text{if } k=i,\\
	E_k(T)\cup {I} &\mbox{if } k=i+1. 
	\end{cases}\]

\begin{example}
Let $n=4$. Taking $I=\{6\}\subseteq E_3(U_4)=\{3,6,8\}$, below we illustrate ${\sf Sl}_{\{6\}}(U_4)$.
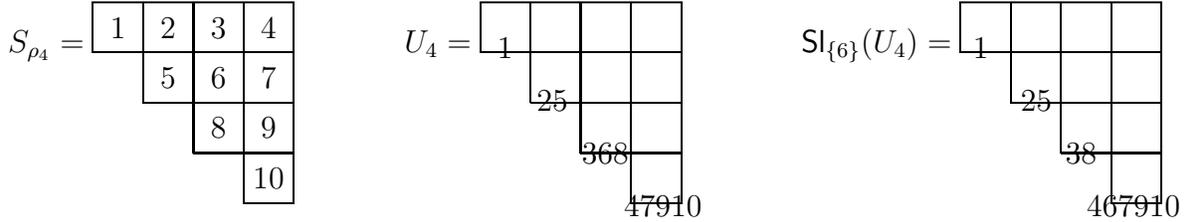
\[\begin{picture}(450,90)
\put(0,65){$S_{\rho_4}=\tableauL{ 1 & 2 & 3 & 4 \\ & 5  & 6 & 7 &\\ &  &  8 &9 \\ & &  & 10 }$
}
\put(150,65){$U_4=\tableauL{ \ & \ & \ & \ \\ & \  & \ & \ \ &\\ &  &  \ &\ \\ & &  & \ }$}
\put(185,63){$1$}
\put(200,43){$2 5$}
\put(217,23){$3 6 8$}
\put(233,03){$4 7 9 10$}

\put(300,65){${\sf Sl}_{\{6\}}(U_4)=\tableauL{ \ & \ & \ & \ \\ & \  & \ & \ \ &\\ &  &  \ &\ \\ & &  & \ }$}
\put(365,63){$1$}
\put(383,43){$2 5$}
\put(400,23){$3 8$}
\put(408,03){$4 6 7 9 10$}
\end{picture}\]      
\end{example}

For $T\in {\sf shSYT}(\nu/\lambda)$, let \[{\sf row}_k(T)=\{\text{entries in row } k \text{ of } T \}.\]
Say $I\subseteq E_{i}(T)$ is \emph{$n$-slidable} if $1\leq i<n$,
\begin{equation}
\label{eqn:n-slidable}
I\subseteq \bigcup_{k=1}^i\{\min (E_{i}(T)\cap {\sf row}_k(S_{\rho_n}))\},
\end{equation}
and for $i< k\leq n$, $E_k(T)=E_k(U_n)$.
\begin{example} Consider $T$ below and $i=3$. Then to the right we have $S_{\rho_4}$ with 
\[\bigcup_{k=1}^3\{\min (E_{3}(T)\cap {\sf row}_k(S_{\rho_4}))\}=\bigcup_{k=1}^3\{\min (\{1,3,5,6,8\}\cap {\sf row}_k(S_{\rho_4}))\}=\{1,5,8\}\] shaded yellow and the remainder of entries of $E_{3}(T)$ shaded gray. Thus any $I\subseteq \{1,5,8\}$ is $4$-slidable, so $\{1,8\}$ is $4$-slidable but $\{1,3,8\}$ is not.
        
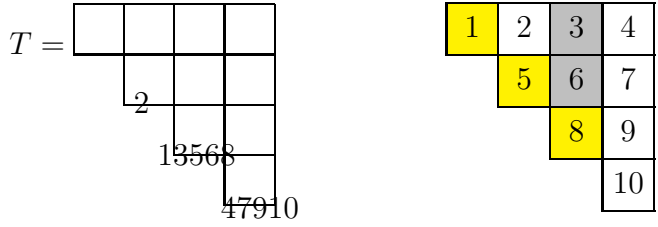
\[\begin{picture}(300,90)
\put(0,65){$T=\tableauL{ \ & \ & \ & \ \\ & \  & \ & \ \ &\\ &  &  \ &\ \\ & &  & \ }$}
\put(47,43){$2$}
\put(56,23){$1 3 5 6 8$}
\put(80,03){$4 7 9 10$}
\put(165,65){$\ytableausetup
{boxsize=1.6em}\begin{ytableau}
 *(yellow)1 & 2  & *(lightgray)3 & 4  \\
\none  & *(yellow)5 & *(lightgray)6 & 7\\
\none  & \none & *(yellow)8 & 9\\
 \none & \none & \none & 10 
\end{ytableau}$}
\end{picture}\]   
\end{example}

The proof of this claim is lengthy and will appear elsewhere:
\begin{claim}\label{lemma:scIFF}
Fix $T\in {\sf shEqSYT}(\rho_n/\rho_n,N)$. Then 
${\sf shEqRect}(T)=S_{\rho_n}$ if and only if $T={\sf Sl}_{I_{n-1}}\circ{\sf Sl}_{I_{n-2}}\circ\ldots\circ{\sf Sl}_{I_1}(U_n)$ where each $I_i\subseteq E_{i}({\sf Sl}_{I_{i-1}}\circ\ldots\circ{\sf Sl}_{I_1}(U_n))$ is $n$-slidable.
\end{claim}

By Claim \ref{lemma:scIFF}, $d_{\rho_n,\rho_n}^{\rho_n}$ equals the number of sequences $\{I_i\}_{i=1}^{n-1}$ where 
\[I_i\subseteq E_{i}({\sf Sl}_{I_{i-1}}\circ\ldots\circ{\sf Sl}_{I_1}(U_n))\]
is $n$-slidable. 
We assert that \[i=\#\bigcup_{k=1}^i\{\min (E_{i}({\sf Sl}_{I_{i-1}}\circ\ldots\circ{\sf Sl}_{I_1}(U_n))\cap {\sf row}_k(S_{\rho_n}))\}.\]
Indeed, to see this, note that $i=\#\bigcup_{k=1}^i \min(E_i(U_n)\cap {\sf row}_k(S_{\rho_n}))\}$ and, by definition of $I$-slidable,
$E_i({\sf Sl}_{I_{i-1}}\circ\ldots\circ{\sf Sl}_{I_1}(U_n))\supseteq E_i(U_n)$. Hence, 
by (\ref{eqn:n-slidable}),
there are $2^i$ choices for each $n$-slidable $I_i$, 
so $d_{\rho_n,\rho_n}^{\rho_n}=2^{\binom{n}{2}}$, as desired.
\qed

We illustrate Claim \ref{lemma:scIFF} with the following example:

\begin{example}\label{ex:slideIllustration} Below is $T={\sf Sl}_{I_{3}}\circ{\sf Sl}_{I_{2}}\circ{\sf Sl}_{I_1}(U_4)$ with the choices of $I_i$ given above each arrow. Beneath each arrow, entries in 
$\bigcup_{k=1}^i\{\min(E_{i}({\sf Sl}_{I_{i-1}}\circ\ldots\circ{\sf Sl}_{I_1}(U_4))\cap {\sf row}_k(S_{\rho_n}))\}$ are shaded yellow in $S_{\rho_4}$ and the remaining entries of $E_i(T)$ are shaded gray. Thus $I_i$ is $4$-slidable if and only if all entries of $I_i$ are yellow. Thus in the example below, $I_1,I_2,$ and $I_3$ are all $4$-slidable. Therefore by Lemma \ref{lemma:scIFF}, ${\sf shEqRect}(T)=S_{\rho_4}$. 
    	
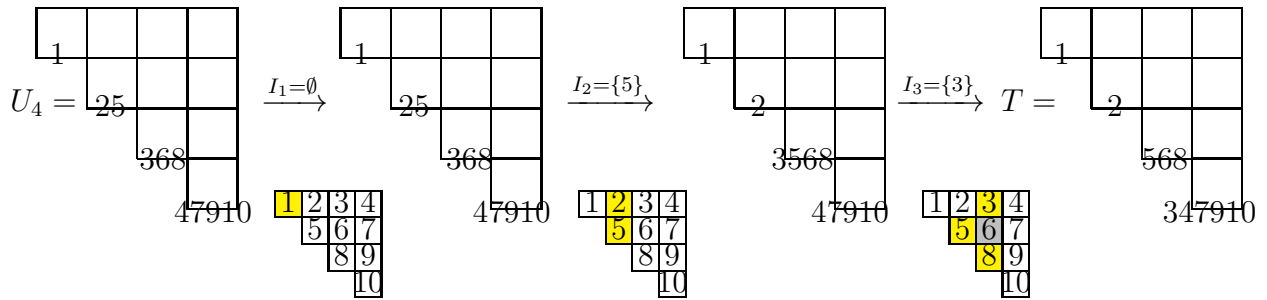
\[\begin{picture}(500,120)
\put(0,75){$U_4=$}
\put(10,95){$\tableauL{ \ & \ & \ & \ \\ & \  & \ & \ \ &\\ &  &  \ &\ \\ & &  & \ }$}
\put(15,93){$1$}
\put(32,73){$2 5$}
\put(49,53){$3 6 8$}
\put(62,33){$4 7 9 10$}
\put(95,75){$\xrightarrow{I_1=\emptyset}$}
\put(100,35){$\ytableausetup
{boxsize=0.8em}\begin{ytableau}
*(yellow)1  & 2  & 3 & 4  \\
\none  & 5 & 6 & 7\\
\none  & \none & 8 & 9\\
 \none & \none & \none & 10 
\end{ytableau}$}
\put(125,95){$\tableauL{ \ & \ & \ & \ \\ & \  & \ & \ \ &\\ &  &  \ &\ \\ & &  & \ }$}
\put(130,93){$1$}
\put(147,73){$2 5$}
\put(165,53){$3 6 8$}
\put(175,33){$4 7 9 10$}
\put(210,75){$\xrightarrow{I_2=\{5\}}$}
\put(215,35){$\begin{ytableau}
1  & *(yellow)2  & 3 & 4  \\
\none  & *(yellow)5 & 6 & 7\\
\none  & \none & 8 & 9\\
 \none & \none & \none & 10 
\end{ytableau}$}
\put(255,95){$\tableauL{ \ & \ & \ & \ \\ & \  & \ & \ \ &\\ &  &  \ &\ \\ & &  & \ }$}
\put(260,93){$1$}
\put(280,73){$2$}
\put(288,53){$3 5 6 8$}
\put(303,33){$4 7 9 10$}
\put(335,75){$\xrightarrow{I_3=\{3\}}$ \  $T=$}
\put(345,35){$\begin{ytableau}
1  & 2  & *(yellow)3 & 4  \\
\none  & *(yellow)5 & *(lightgray)6 & 7\\
\none  & \none & *(yellow)8 & 9\\
 \none & \none & \none & 10 
\end{ytableau}$}
\put(390,95){$\tableauL{ \ & \ & \ & \ \\ & \  & \ & \ \ &\\ &  &  \ &\ \\ & &  & \ }$}
\put(395,93){$1$}
\put(415,73){$2$}
\put(428,53){$5 6 8$}
\put(436,33){$3 4 7 9 10$}
\end{picture}\]    

However, in the example below, $I'_1,I'_2$ are $4$-slidable, but $I'_3$ is not. Thus by Claim \ref{lemma:scIFF}, ${\sf shEqRect}(T')\neq S_{\rho_4}$.
 	
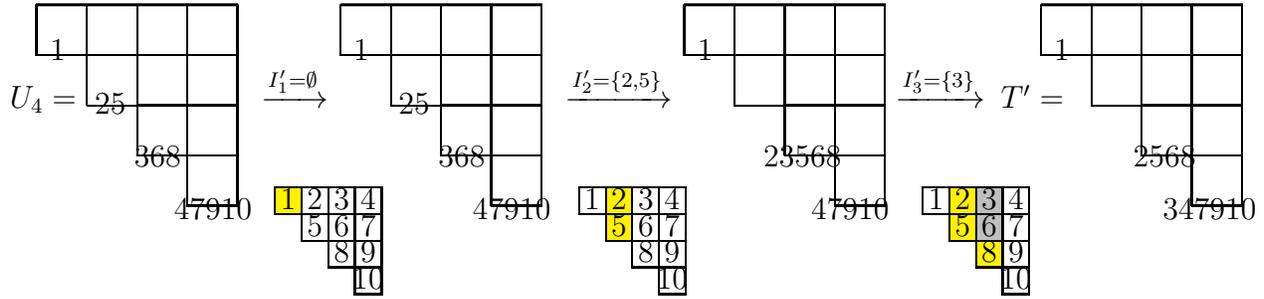
\[\begin{picture}(500,120)
\put(0,75){$U_4=$}
\put(10,95){$\tableauL{ \ & \ & \ & \ \\ & \  & \ & \ \ &\\ &  &  \ &\ \\ & &  & \ }$}
\put(15,93){$1$}
\put(32,73){$2 5$}
\put(47,53){$3 6 8$}
\put(62,33){$4 7 9 10$}
\put(95,75){$\xrightarrow{I'_1=\emptyset}$}
\put(100,35){$\ytableausetup
{boxsize=0.8em}\begin{ytableau}
*(yellow)1  & 2  & 3 & 4  \\
\none  & 5 & 6 & 7\\
\none  & \none & 8 & 9\\
 \none & \none & \none & 10 
\end{ytableau}$}
\put(125,95){$\tableauL{ \ & \ & \ & \ \\ & \  & \ & \ \ &\\ &  &  \ &\ \\ & &  & \ }$}
\put(130,93){$1$}
\put(147,73){$2 5$}
\put(162,53){$3 6 8$}
\put(175,33){$4 7 9 10$}
\put(210,75){$\xrightarrow{I'_2=\{2,5\}}$}
\put(215,35){$\begin{ytableau}
1  & *(yellow)2  & 3 & 4  \\
\none  & *(yellow)5 & 6 & 7\\
\none  & \none & 8 & 9\\
 \none & \none & \none & 10 
\end{ytableau}$}
\put(255,95){$\tableauL{ \ & \ & \ & \ \\ & \  & \ & \ \ &\\ &  &  \ &\ \\ & &  & \ }$}
\put(260,93){$1$}
\put(285,53){$2 3 5 6 8$}
\put(303,33){$4 7 9 10$}
\put(335,75){$\xrightarrow{I'_3=\{3\}}$ \  $T'=$}
\put(345,35){$\begin{ytableau}
1  & *(yellow)2  & *(lightgray)3 & 4  \\
\none  & *(yellow)5 & *(lightgray)6 & 7\\
\none  & \none & *(yellow)8 & 9\\
 \none & \none & \none & 10 
\end{ytableau}$}
\put(390,95){$\tableauL{ \ & \ & \ & \ \\ & \  & \ & \ \ &\\ &  &  \ &\ \\ & &  & \ }$}
\put(395,93){$1$}
\put(425,53){$2 5 6 8$}
\put(436,33){$3 4 7 9 10$}
\end{picture}\]      
\end{example}

While $2^{\binom{n}{2}}$ is the number of labeled graphs on $n$ vertices, consulting
the Online Encyclopedia of Integer Sequences \cite{oeis}, one also finds that it counts the number of
\begin{itemize}
\item perfect matchings of order $n$ Aztec diamond \cite{Speyer},
\item Gelfand-Zeitlin patterns with bottom row $[1,2,3,...,n]$ \cite{Zeilberger}, and
\item certain domino tilings \cite[A006125]{oeis}
\end{itemize}
among other things.

We end with a problem of enumerative combinatorics:

\begin{problem}
Give bijections between the shifted edge labeled tableaux counted by $d_{\rho_n,\rho_n}^{\rho_n}$
and the equinumerous objects above.
\end{problem}

\section*{Acknowledgements}
AY's thanks to Bill Fulton goes back to 1999. In May 2018, the authors attended a conference on Schubert calculus held at Ohio State University where
Bill kindly shared the fine points of his ongoing work with David Anderson. This was the stimulus for this chapter. We also thank David Anderson, Soojin Cho, Sergey Fomin, Allen Knutson, Gidon Orelowitz, John Stembridge, Hugh Thomas and Brian Shin for helpful remarks. We thank Anshul Adve, David Anderson, Cara Monical, Oliver Pechenik, Ed Richmond, and Hugh Thomas for their contributions reported here. AY was partially supported by an NSF grant, a UIUC Campus research board grant, and a Simons Collaboration Grant. This material is based upon work of CR supported by the National Science Foundation Graduate Research Fellowship Program under Grant No. DGE -- 1746047.

\end{document}